\documentclass[12pt,reqno]{article}
\usepackage[utf8]{inputenc}
\usepackage{amsmath,amsthm}
\usepackage{fullpage}
\usepackage{booktabs}
\usepackage{graphicx}
\usepackage{float}
\usepackage{xcolor}
\usepackage[export]{adjustbox}
\usepackage{enumitem}
\definecolor{webgreen}{rgb}{0,.5,0}
\definecolor{webbrown}{rgb}{.6,0,0}
\usepackage[colorlinks=true,
linkcolor=webgreen,
filecolor=webbrown,
citecolor=webgreen]{hyperref}
\usepackage[list=true]{subcaption}
\everymath{\displaystyle}

\theoremstyle{plain}
\newtheorem{theorem}{Theorem}
\newtheorem{lemma}[theorem]{Lemma}
\newtheorem{corollary}[theorem]{Corollary}

\theoremstyle{definition}
\newtheorem{definition}[theorem]{Definition}
\newtheorem{example}[theorem]{Example}
\newtheorem{notation}[theorem]{Notation}
\newtheorem{remark}[theorem]{Remark}

\DeclareMathOperator{\bra}{bp}
\newcommand{\cla}[1]{\pi\left(#1\right)}
\newcommand{\clb}[1]{\pi\big(#1\big)}
\newcommand{\cof}[1]{\left[#1\right]}
\newcommand{\bp}[1]{\bra\left(#1\right)}
\newcommand{\br}[1]{\left<#1\right>}

\newcommand{\iD}[1]{\item[D\ref*{#1}:]}
\newcommand{\iN}[1]{\item[N\ref*{#1}:]}
\newcommand{\iR}[1]{\item[R\ref*{#1}:]}
\newcommand{\iB}[1]{\item[B\ref*{#1}:]}
\newcommand{\Figs}[1]{\hyperref[#1]{Figure~\ref*{#1}}}
\newcommand{\seqnum}[1]{\href{http://oeis.org/#1}{\underline{#1}}}
\newcommand{\Tabs}[1]{\hyperref[#1]{Table~\ref*{#1}}}
\newcommand{\Lem}[1]{\hyperref[#1]{Lemma~\ref*{#1}}}
\newcommand{\Rem}[1]{\hyperref[#1]{Remark~\ref*{#1}}}
\newcommand{\Subsec}[1]{\hyperref[#1]{subsection~\ref*{#1}}}
\newcommand{\Sec}[1]{\hyperref[#1]{section~\ref*{#1}}}

\title{\bf A bracket polynomial for $ 2 $-tangle shadows}
\author{Franck  Ramaharo\\
\small D\'epartement de Math\'ematiques et Informatique\\[-0.8ex]
\small Universit\'e d'Antananarivo\\[-0.8ex] 
\small 101 Antananarivo, Madagascar\\
\small\href{mailto:franck.ramaharo@gmail.com}{\tt franck.ramaharo@gmail.com}\\
}

\date{\small\today}

\begin{document}
\maketitle

\begin{abstract}
We compute the Kauffman bracket polynomial of the numerator and denominator closures of $ A + A + \cdots + A $ ($ A $ is repeated $ n $ times), where $ A $ is a $ 2 $-tangle shadow that has at most $ 4 $ crossings.

\bigskip\noindent {Keywords:} tangle shadow, bracket polynomial.
\end{abstract}

\section{Introduction}
A $ 2 $-tangle shadow is a $ 2 $-tangle diagram without under/over crossing information \cite{Krebes,Medina}. The following definition holds for such class of tangles.

\begin{definition} 
Let $ A:=\protect\includegraphics[width=.05\linewidth,valign=c]{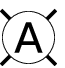} $ and $ B:=\protect\includegraphics[width=.05\linewidth,valign=c]{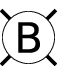} $ be two $ 2 $-tangles. Then

\begin{itemize}
\item $ A+B:=\protect\includegraphics[width=.11\linewidth,valign=c]{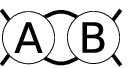} $ denotes the \textit{horizontal sum};
\item $ A*B:=\protect\includegraphics[width=.05\linewidth,valign=c]{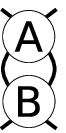} $  denotes the \textit{vertical sum};
\item $ \dfrac{1}{A}:=\protect\includegraphics[width=.05\linewidth,valign=c]{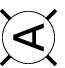} $ denote the \textit{inverse of $ A  $} which is obtained  by turning the tangle counter-clockwise by $ 90 $ degree  in the plane;
\item $ N(A):= \protect\includegraphics[width=.075\linewidth,valign=c]{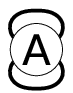}$ and 
$ D(A):= \protect\includegraphics[width=.075\linewidth,valign=c]{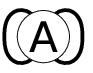} $ denote the \textit{numerator} and the \textit{denominator closures}, respectively, of $ A $.
\end{itemize}

The simplest $ 2 $-tangles are
\[
[0]:=\protect\includegraphics[width=.05\linewidth,valign=c]{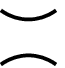},\qquad
[\infty]:=\protect\includegraphics[width=.05\linewidth,valign=c]{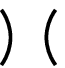},\qquad
[n]:=\protect\includegraphics[width=.2\linewidth,valign=c]{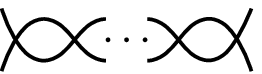},\qquad \dfrac{1}{[n]}:=\protect\includegraphics[width=.05\linewidth,valign=c]{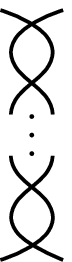}. 
\]
\end{definition}

Set $ A_n:=A+A+\cdots+A $, where $ A $ is repeated $ n $ times (with $ A_0:=[0] $). In the present paper, we compute the Kauffman bracket polynomial of the closures $ N\left(A_n\right) $ and $ D\left(A_n\right) $. We restrict our computation to $ 2 $-tangles of up to $ 4 $ crossings. In fact, since $ N\left(A_n\right) $ and $ D\left(A_n\right) $ are knot shadows, we extend previous results in which we investigated the generating polynomial of knots of up to $ 3 $ crossings \cite{Ramaharo}.

The rest of the paper is organized as follows. We give in \Sec{sec:bracket} some properties of the bracket polynomial of $ 2 $-tangle shadows. By those properties, we obtain a set of integer polynomials which is presented in \Sec{sec:application}. We give as well the tables giving the coefficients in the expansion of each polynomial. 

Throughout the paper by ``tangle'' and ``knot'' we understand ``2-tangle shadow diagram'' and ``knot shadow diagram'', respectively. Also, we study tangles and knots up to planar isotopy.

\section{Definition and construction}\label{sec:bracket}

In the present framework, the Kauffman bracket polynomial of a shadow diagram is an integer polynomial in $ x $ such that the coefficient of $ x^k $ matches the number of state diagrams having exactly $ k $ circles \cite{Ramaharo,Ramaharo1}. This definition extends to tangles as follows.

\begin{definition}
Let $ A $ be a tangle. The bracket polynomial of $ A $, denoted $ \left<A\right> $, is defined by
\begin{equation} \label{eq:lincomb}
\br{A} = a(A)\br{[0]} + b(A)\br{[\infty]},
\end{equation}
where $ a(A) $ and $ b(A) $ are integer polynomials. The linear combination of $ \br{[0]} $ and $ \br{[\infty]} $ in \eqref{eq:lincomb} is obtained from the following usual rules for shadow diagrams:
\begin{itemize}
\item $ \br{\bigcirc} = x $;
\item $ \br{\bigcirc\sqcup A} = x\br{A} $;
\item  $\left<\protect\includegraphics[width=.04\linewidth,valign=c]{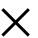}\right>=\left<\protect\includegraphics[width=.04\linewidth,valign=c]{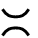}\right>+\left<\protect\includegraphics[width=.04\linewidth,valign=c]{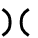}\right>$, where \protect\includegraphics[width=.04\linewidth,valign=c]{split_2} and  \protect\includegraphics[width=.04\linewidth,valign=c]{split_1} denote the states of the crossing \protect\includegraphics[width=.04\linewidth,valign=c]{crossing}.
\end{itemize}
\end{definition}

For example, $ \left<[2]\right>=\left<[0]\right>+(x+2)\left<[\infty]\right> $. Indeed, we have
\begin{align*}
\left<\protect\includegraphics[width=.075\linewidth,valign=c]{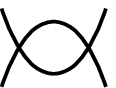}\right> &= \left<\protect\includegraphics[width=.075\linewidth,valign=c]{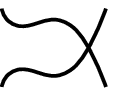}\right>+\left<\protect\includegraphics[width=.075\linewidth,valign=c]{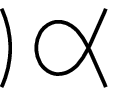}\right>\\
&=\left<\protect\includegraphics[width=.075\linewidth,valign=c]{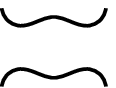}\right>+\left<\protect\includegraphics[width=.075\linewidth,valign=c]{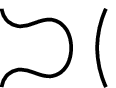}\right>+\left<\protect\includegraphics[width=.075\linewidth,valign=c]{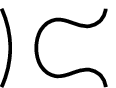}\right>+\left<\protect\includegraphics[width=.075\linewidth,valign=c]{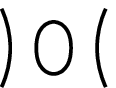}\right>.
\end{align*}

\begin{lemma}
The bracket of the closures $ N(A) $ and $ D(A) $ are given by
\begin{equation*}
\br{N(A)} = x^2a(A) + xb(A)   
\end{equation*}
and
\begin{equation*}
\br{D(A)} = xa(A) + x^2b(A),
\end{equation*}
respectively.
\end{lemma}
\begin{proof}
The crossings are smoothed without affecting the closures \cite{KauffmanLambropoulou}. Hence
\begin{align*}
\br{N(A)} &= a(A)\br{N([0])} + b(A)\br{N([\infty])} \\
&= a(A)x^2 + b(A)x 
\end{align*}
and
\begin{align*}
\br{D(A)} &= a(A)\br{D([0])} + b(A)\br{D([\infty])} \\
&= a(A)x + b(A)x^2.
\end{align*}
\end{proof}

\begin{remark} Let $ R(A)$ be the operation defined by $N\left(A+[1]\right)$. Then
\[
\br{R(A)} = \br{N(A)} + \br{D(A)} = \left(x^2 + x\right)\left(a(A) + b(A)\right).
\]
We understand $ R(A) $  as an interpretation of  the DNA recombination \cite{ErnstSumners,KauffmanLambropoulou}. For simplicity, we call it the $ R $-closure.
\end{remark}

\begin{lemma}
Given two tangles $ A $ and $ B $, 
\begin{equation}\label{eq:bracket_sum}
\br{A+B} = a(A)a(B)\br{[0]} + \big(a(A)b(B) + b(A)a(B) + xb(A)b(B)\big)\br{[\infty]}.
\end{equation}

\end{lemma}
\begin{proof} First note that for any tangle $ A $, 
\begin{itemize}
\item $ A + [0] = [0] + A = A$ ,
\item $ \br{A + [\infty]} = \br{[\infty] + A} = \left(a(A) + xb(A)\right)\br{[\infty]} $.
\end{itemize}
Then we establish \eqref{eq:bracket_sum} by computing the states of $ A $ leaving $ B $ intact, and then those of $ B $:
\begin{align*}
\br{A+B} & = a(A)\br{[0] + B} + b(A)\br{[\infty] + B}\\
& = a(A)\left(a(B)\br{[0]} + b(B)\br{[\infty]}\right) + b(A)\left(a(B) + xb(B)\right)\br{[\infty]}\\
& = a(A)a(B)\br{[0]} + \big(a(A)b(B) + b(A)a(B) + xb(A)b(B)\big)\br{[\infty]}.
\end{align*}
\end{proof}
We now use similar notations to those of Kauffman \cite[p.\ 88]{Kauffman1}: we set $ A^N := \br{N(A)} $, $ A^D := \br{D(A)} $ and $ A^R:=\br{R(A)} $. The following identities hold.

\begin{lemma}
If $ A $ and $ B $ are tangles, then
\begin{align*} 
x(A + B)^D & = A^D B^D;\\
x\left(x^2 - 1\right)(A + B)^N & = \left(A^N B^N + A^D B^D\right)x - \left(A^D B^N + A^N B^D\right);\\
x\left(x^2 - 1\right)(A + B)^R & = A^D B^D x^2 + \left(A^N B^N + A^D B^D\right) x - \left(A^D B^N + A^N B^D + A^D B^D\right).
\end{align*}
\end{lemma}

Let us now refer to the pair $\begin{pmatrix}
a(A)\\
b(A)  
\end{pmatrix} $,  denoted $ \bp{A}$,  as  \textit{bracket pair} of $ A $ \cite{EliahouFromentin}, and let $ M(A)$ be the matrix defined as $\begin{pmatrix}
a(A) & 0\\
b(A) & a(A) + xb(A)   
\end{pmatrix} $. 

For example $ \bp{[0]} = \begin{pmatrix}
1 \\
0
\end{pmatrix} $,  $ M([0]) = \begin{pmatrix}
1 & 0\\
0 & 1   
\end{pmatrix} $, $ M([\infty]) = \begin{pmatrix}
0 \\1
\end{pmatrix} $ and $ M([\infty]) = \begin{pmatrix}
0 & 0\\
1 & x  
\end{pmatrix} $. 

\begin{lemma}\label{lem:brAn}
With previous notations and usual matrix algebra we have 
\begin{itemize}
\item $ \bp{A} = M(A)\bp{[0]}$;

\item $ \bp{A + B} = M(A)\bra{(B)} = M(A)M(B)\bp{[0]} $.
\end{itemize}
\end{lemma}

We use \Lem{lem:brAn} to show by induction that
\[
\bp{A_n} = \big(M(A)\big)^n\bp{[0]},
\]
where  
\begin{equation*}
\big(M(A)\big)^n = M(A_n) = \begin{pmatrix}
\big(a(A)\big)^n & 0\\
\dfrac{\big(a(A) + xb(A)\big)^n - \big(a(A)\big)^n}{x} & \big(a(A) + xb(A)\big)^n
\end{pmatrix}.
\end{equation*} 

\begin{theorem} 
If $ \br{A_n} = a(A_n)\br{[0]} +  b(A_n)\br{[\infty]}$, then 
\[
a(A_n) = \big(a(A)\big)^n
\]
and 
\[
b(A_n) = \dfrac{\big(a(A) + xb(A)\big)^n - \big(a(A)\big)^n}{x}.
\]
\end{theorem}

Formulas for the closures are then immediate:
\begin{corollary}
The closures of $ A_n $ verify
\begin{align*} 
\left(A_{n}\right)^D & = x\big(a(A) + xb(A)\big)^{n};\\
\left(A_{n}\right)^N & = \big(a(A) + xb(A)\big)^n + \left(x^2 - 1\right)\big(a(A)\big)^n;\\
\left(A_{n}\right)^R & = (x + 1)\big(a(A) + xb(A)\big)^n + \left(x^2 - 1\right)\big(a(A)\big)^n.
\end{align*}
\end{corollary}

For example, knowing that $ \br{[1]} = \br{[0]} + \br{[\infty]} $ and $ [n] = [1] + [1] + \cdots + [1] $,  we have 
\begin{align*}
[n]^D & = x(x + 1)^n;\\ 
[n]^N & = (x + 1)^n + x^2 - 1;\\
[n]^R & = (x + 1)^{n + 1} + x^2 - 1.
\end{align*}

\begin{definition}[\cite{KauffmanLambropoulou}]
We define the \textit{polynomial fraction}, $ F(A) $, of the tangle $ A $ as
\begin{equation}
F(A) := \dfrac{b(A)}{a(A)}.
\end{equation}
\end{definition}

With the previous notations, we have
\begin{itemize}
\item $ F(A + B) = F(A) + F(B) + xF(A)F(B) $;
\item $ F\left(\dfrac{1}{A}\right) = \dfrac{1}{F(A)} $;
\item $ F(A_n) = \dfrac{\big(1 + xF(A)\big)^n - 1}{x} $.
\end{itemize}

\begin{remark}\label{Rem:bracketequal1}
Let $A$ and $B$ be two tangles. We have $\br{A} = \br{B}$ 
\begin{enumerate}[label=(\roman*)]
\item if $A$ is planar isotopic to $B$, or\label{it:i}
\item if there exist a tangle $T$ and two knots $K$ and $K'$, with $\br{K} = \br{K'}$, such that $\br{A} = \br{T\#K}
$ and $\br{B} = \br{T\#K'}$.\label{it:ii}
\end{enumerate}
By ``tangle $\#$ knot'' we mean the knot is connected at some segment of the tangle, $\#$ denoting the usual connected sum. Recall that for two knots $ K $ and $ K' $ \cite{Ramaharo}, 
\begin{equation}\label{eq:bracket-sum}
\br{K\#K'} = x^{-1}\br{K}\br{K'}.
\end{equation}
Here, \ref{it:i}  and \ref{it:ii}  suggest that there exists a tangle $A$ for which the identity $\br{A}=\br{T\#K}$ is satisfied only if $A$ is planar isotopic to $T$, and $K$ is the unknot $\bigcirc$. For such case, we say that $A$ is \textit{prime}, and \textit{locally knotted} otherwise. If $ \br{A}=x^{-1}\br{T^*}\br{K} $ and $ T^* $ is prime, then we say that $T^*$ is the \textit{skeleton} of $A$.
\end{remark}

The polynomial fraction allows us to ``extract'' the skeleton of a tangle. Indeed, assume that $ \br{A}=x^{-1}\br{T^*}\br{K} $. By \eqref{eq:bracket-sum} we obtain a common factor which affects both  brackets $ \br{[0]} $ and $ \br{[\infty]} $, i.e.,
\[
\br{A} = \left(x^{-1}\br{K}a(T^*)\right)\br{[0]} + \left(x^{-1}\br{K}b(T^*)\right)\br{[\infty]},
\]
and consequently 
\[
F(A) = F(T^*).
\] 
Last equality  implies that the skeleton $T$ verifies $\gcd(a(T^*), b(T^*)) = 1$. For example, consider tangle $ A $ as in \Figs{subfig:A}.  We have
\[
F(A) = \dfrac{x^2 + 4x + 3}{x^2 + 4x + 3} = 1 = F([1]).
\]

\begin{figure}[!ht]
\centering
\hspace*{\fill}
\subcaptionbox{$ A $\label{subfig:A}}{\includegraphics[width=0.4\linewidth]{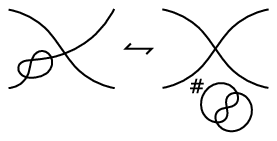}}%
\hspace*{\fill}
\subcaptionbox{$ B $\label{subfig:B}}{\includegraphics[width=0.4\linewidth]{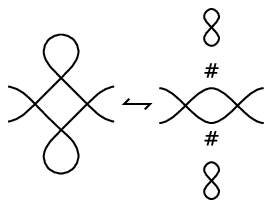}}%
\hspace*{\fill}
\caption{Locally knotted tangles.}
\label{Fig:prime_knotted}
\end{figure}

\begin{remark}\label{Rem:bracketequal2}
Assume that we can disconnect more than one knot in the tangle $A$, i.e.,
\[
A = \left(\cdots\left(\left(T\#K^{(1)}\right)\#K^{(2)}\right)\#\cdots\right)\# K^{(m)},
\]
Then by \eqref{eq:bracket-sum} we have 
\[\br{A} = x^{-m+1}\br{T}\br{K^{(1)}}\cdots\br{K^{(m)}} = x^{-1}\br{T}\br{K}\] 
which is also the bracket polynomial of a certain $  T\#{K} $, where $ K = K^{(1)}\#\cdots\#K^{(m)} $. 
\end{remark}

\begin{example}
Tangle $ B $ in \Figs{subfig:B}  has the same bracket polynomial as the tangle belows
\[
\protect\includegraphics[width=0.2\linewidth,valign=c]{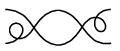},\qquad \protect\includegraphics[width=0.2\linewidth,valign=c]{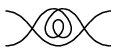},\qquad \protect\includegraphics[width=0.2\linewidth,valign=c]{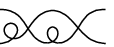},
\]
and verifies
\[
F(B) = \dfrac{x^3 + 4x^2 + 5x + 2}{x^2 + 2x +1} = x + 2 = F([2]).
\]
\end{example}

\section{Results}\label{sec:application}
\subsection{Classifying of knots and tangles}\label{subsec:presentation}
Let $ c(X) $ denote the number of crossings in the diagram $ X $. Our goal is to find all possible values of $ \br{A_n} $ for any tangle  $ A $ with $ c(A)\leq 4 $. Since $ \br{A}$ can always be written as $ x^{-1}\br{T}\br{K} $, with $ T $ prime, it suffices to find all values of $ \br{T} $ and $ \br{K} $ when $c(A)\in\{1,2,3,4\} $ and $ c(A)= c(T)+c(K)  $. As seen in \Rem{Rem:bracketequal1} and \ref{Rem:bracketequal2}, two different diagrams can have the same bracket polynomial. Let us then introduce the following equivalence relation which allows us to to gather together tangles and knots that share the same bracket expression.

\begin{definition}
We say that $ X $ is equivalent to $ X' $  if $ \br{X} = \br{X'} $. We let $ \cla{X} $ denote the equivalence class of $ X $.
\end{definition}

For each numbered equivalence classes listed below, \textbf{Ai} is a representative of a class, \textbf{Bi} is the entry for the bracket of \textbf{Ai}, and entries \textbf{Di}, \textbf{Ni} and \textbf{Ri} are for the brackets of the denominator, numerator  and $ R $ closures of $ \left(\textbf{Ai}\right)_n $, respectively. 

We begin with the classes of prime tangles (cf. \Figs{Fig:prime_tangles}).

\begin{figure}[!ht]
\centering

\subcaptionbox{$ \pi(\textbf{A}\ref{tangle:A1}) $}{\includegraphics[width=0.16\linewidth]{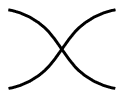}\label{subfig:A1}}%
\subcaptionbox{$ \pi(\textbf{A}\ref{tangle:A2}) $}{\includegraphics[width=0.16\linewidth]{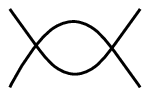}\label{subfig:A2}}%
\subcaptionbox{$ \pi(\textbf{A}\ref{tangle:A3}) $}{\includegraphics[width=0.16\linewidth]{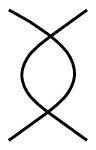}\label{subfig:A3}}%
\subcaptionbox{$ \pi(\textbf{A}\ref{tangle:A4}) $}{\includegraphics[width=0.16\linewidth]{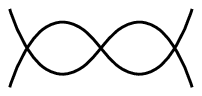}\label{subfig:A4}}%
\subcaptionbox{$ \pi(\textbf{A}\ref{tangle:A5}) $}{\includegraphics[width=0.16\linewidth]{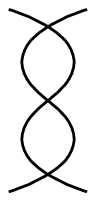}\label{subfig:A5}}%

\subcaptionbox{$ \pi(\textbf{A}\ref{tangle:A6}) $}{\includegraphics[width=0.165\linewidth]{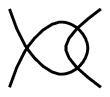}\label{subfig:A6}}
\subcaptionbox{$ \pi(\textbf{A}\ref{tangle:A7}) $}{\includegraphics[width=0.165\linewidth]{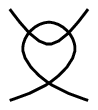}\label{subfig:A7}}%
\subcaptionbox{$ \pi(\textbf{A}\ref{tangle:A8}) $}{\includegraphics[width=0.165\linewidth]{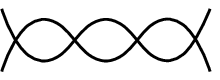}\label{subfig:A8}}%
\subcaptionbox{$ \pi(\textbf{A}\ref{tangle:A9}) $}{\includegraphics[width=0.165\linewidth]{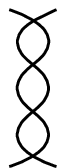}\label{subfig:A9}}%
\subcaptionbox{$ \pi(\textbf{A}\ref{tangle:A10}) $}{\includegraphics[width=0.165\linewidth]{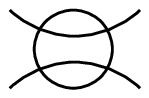}\label{subfig:A10}}%
\subcaptionbox{$ \pi(\textbf{A}\ref{tangle:A11}) $}{\includegraphics[width=0.165\linewidth]{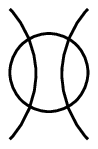}\label{subfig:A11}}%

\hspace*{\fill}
\subcaptionbox{$ \pi(\textbf{A}\ref{tangle:A12}) $}{\includegraphics[width=0.49\linewidth]{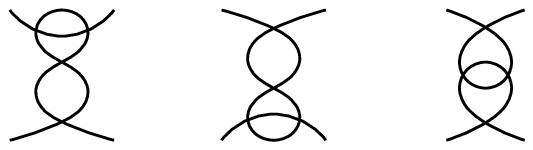}\label{subfig:A12}}%
\subcaptionbox{$ \pi(\textbf{A}\ref{tangle:A13}) $}{\includegraphics[width=0.49\linewidth]{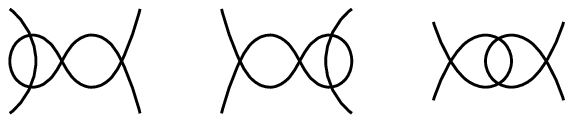}\label{subfig:A13}}%
\hspace*{\fill}\\

\subcaptionbox{$ \pi(\textbf{A}\ref{tangle:A14}) $}{\includegraphics[width=0.8\linewidth]{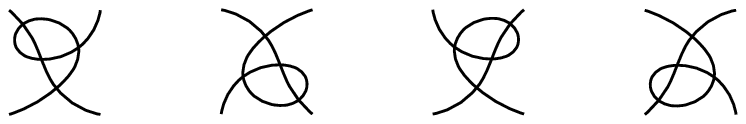}\label{subfig:A14}}%

\subcaptionbox{$ \pi(\textbf{A}\ref{tangle:A15}) $}{\includegraphics[width=0.8\linewidth]{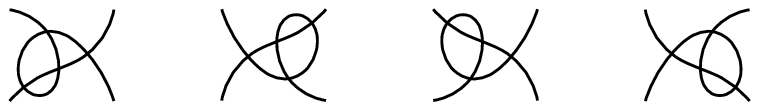}\label{subfig:A15}}%

\hspace*{\fill}
\subcaptionbox{$ \pi(\textbf{A}\ref{tangle:A16}) $}{\includegraphics[width=0.375\linewidth]{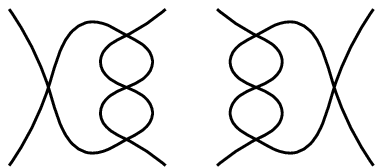}\label{subfig:A16}}%
\hspace*{\fill}
\subcaptionbox{$ \pi(\textbf{A}\ref{tangle:A17}) $}{\includegraphics[width=0.375\linewidth]{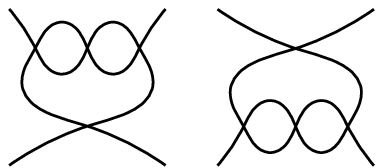}\label{subfig:A17}}%
\hspace*{\fill}
\caption[]{Prime tangles of up to $ 4 $ crossings.}
\label{Fig:prime_tangles}
\end{figure}

\begin{enumerate}

\item $ \clb{[1]} = \big\{[1]\big\} $. \label{tangle:A1}
\begin{description}
\iB{tangle:A1} $\br{[0]}+\br{[\infty]} $.
\iD{tangle:A1} $ x(x+1)^n $, \Tabs{Tab:T1}.
\iN{tangle:A1} $(x+1)^n +x^2 -1$, \Tabs{Tab:T2}.
\iR{tangle:A1} $(x+1)^{n+1} +x^2 -1$, \Tabs{Tab:T3}.
\end{description}

\item $ \clb{[2]} = \big\{[2]\big\} $. \label{tangle:A2}
\begin{description}
\iB{tangle:A2} $\br{[0]}+(x+2)\br{[\infty]} $.
\iD{tangle:A2} $x\left(x+1\right)^{2n} $, \Tabs{Tab:T4} (cf.\ \textbf{D}\ref{tangle:A18}).
\iN{tangle:A2} $\left(x+1\right)^{2n} +x^2-1$, \Tabs{Tab:T5}.
\iR{tangle:A2} $\left(x+1\right)^{2n+1} +x^2-1$, \Tabs{Tab:T6}.
\end{description}

\item $ \cla{\dfrac{1}{[2]}} = \left\{\dfrac{1}{[2]}\right\} $. \label{tangle:A3}
\begin{description}
\iB{tangle:A3} $(x+2)\br{[0]} +\br{[\infty]} $.
\iD{tangle:A3} $x(2x+2)^{n} $, \Tabs{Tab:T7}.
\iN{tangle:A3} $(2x+2)^{n} +\left(x^2-1\right)(x+2)^n$, \Tabs{Tab:T8}.
\iR{tangle:A3} $(x+1)(2x+2)^{n} +\left(x^2-1\right)(x+2)^n$, \Tabs{Tab:T9}.
\end{description}

\item $ \clb{[3]} = \big\{[3]\big\} $. \label{tangle:A4}
\begin{description}
\iB{tangle:A4} $\br{[0]} +\left(x^2+3x+3\right)\br{[\infty]} $.
\iD{tangle:A4} $x\left(x+1\right)^{3n} $, \Tabs{Tab:T10} (cf.\ \textbf{D}\ref{tangle:A25}, \textbf{D}\ref{tangle:A19}).
\iN{tangle:A4} $\left(x+1\right)^{3n} +x^2-1$, \Tabs{Tab:T11}.
\iR{tangle:A1} $\left(x+1\right)^{3n+1} +x^2-1$, \Tabs{Tab:T12}.
\end{description}

\item $ \cla{\dfrac{1}{[3]}} = \left\{\dfrac{1}{[3]}\right\} $. \label{tangle:A5}
\begin{description}
\iB{tangle:A5} $\left(x^3+3x+3\right)\br{[0]} +\br{[\infty]} $.
\iD{tangle:A5} $x\left(x^2+4x+3\right)^{n} $, \Tabs{Tab:T13} (cf.\ \textbf{D}\ref{tangle:A7}).
\iN{tangle:A5} $\left(x^2+4x+3\right)^{n}+\left(x^2-1\right)\left(x^2+3x+3\right)^n$, \Tabs{Tab:T14}.
\iR{tangle:A5} $(x+1)\left(x^2+4x+3\right)^{n}+\left(x^2-1\right)\left(x^2+3x+3\right)^n$, \Tabs{Tab:T15}.
\end{description}

\item $ \cla{[1]+\dfrac{1}{[2]}} = \left\{[1]+\dfrac{1}{[2]},\dfrac{1}{[2]}+[1]\right\} $.  \label{tangle:A6}
\begin{description}
\iB{tangle:A6} $ (x+2)\br{[0]} +(2x+3)\br{[\infty]}  $.
\iD{tangle:A6} $x\left(2x^2+4x+2\right)^{n} $, \Tabs{Tab:T16} (cf.\ \textbf{D}\ref{tangle:A28}, \textbf{D}\ref{tangle:A20}).
\iN{tangle:A6} $\left(2x^2+4x+2\right)^{n}+\left(x^2-1\right)(x+2)^n$, \Tabs{Tab:T17}.
\iR{tangle:A6} $(x+1)\left(2x^2+4x+2\right)^{n}+\left(x^2-1\right)(x+2)^n$, \Tabs{Tab:T18}.
\end{description}

\item $ \clb{[1]*[2]} = \big\{[1]*[2],[2]*[1]\big\} $. \label{tangle:A7}
\begin{description}
\iB{tangle:A7} $ (2x+3)\br{[0]} +(x+2)\br{[\infty]} $.
\iD{tangle:A7} $x\left(x^2+4x+3\right)^{n} $, \Tabs{Tab:T13} (cf.\ \textbf{D}\ref{tangle:A5}).
\iN{tangle:A7} $\left(x^2+4x+3\right)^{n}+\left(x^2-1\right)(2x+3)^n$, \Tabs{Tab:T19}.
\iR{tangle:A7} $(x+1)\left(x^2+4x+3\right)^{n}+\left(x^2-1\right)(2x+3)^n$, \Tabs{Tab:T20}.
\end{description}

\item $ \clb{[4]} = \big\{[4]\big\} $. \label{tangle:A8}
\begin{description}
\iB{tangle:A8} $\br{[0]} +\left(x^3+4x^2+6x+4\right)\br{[\infty]} $.
\iD{tangle:A8} $x\left(x+1\right)^{4n} $, \Tabs{Tab:T21}  (cf.\ \textbf{D}\ref{tangle:A21}, \textbf{D}\ref{tangle:A26}, \textbf{D}\ref{tangle:A31}).
\iN{tangle:A8} $\left(x+1\right)^{4n} +x^2-1$, \Tabs{Tab:T22}.
\iR{tangle:A8} $\left(x+1\right)^{4n+1} +x^2-1$, \Tabs{Tab:T23}.
\end{description}

\item $ \cla{\dfrac{1}{[4]}} = \left\{\dfrac{1}{[4]}\right\} $. \label{tangle:A9}
\begin{description}
\iB{tangle:A9} $\left(x^3+4x^2+6x+4\right)\br{[0]} +\br{[\infty]}$.
\iD{tangle:A9} $x\left(x^3+4x^2+7x+4\right)^{n} $, \Tabs{Tab:T24} (cf.\ \textbf{D}\ref{tangle:A10}, \textbf{D}\ref{tangle:A17}).
\iN{tangle:A9} $\left(x^3+4x^2+7x+4\right)^{n}+\left(x^2-1\right)\left(x^3+4x^2+6x+4\right)^n$, \Tabs{Tab:T25}.
\iR{tangle:A9} $(x+1)\left(x^3+4x^2+7x+4\right)^{n}+\left(x^2-1\right)\left(x^3+4x^2+6x+4\right)^n$, \Tabs{Tab:T26}.
\end{description}

\item $ \clb{[2]*[2]} = \big\{[2]*[2]\big\} $. \label{tangle:A10}
\begin{description}
\iB{tangle:A10} $(3x+4)\br{[0]} +\left(x^2+4x+4\right)\br{[\infty]}$.
\iD{tangle:A10} $x\left(x^3+4x^2+7x+4\right)^{n} $, \Tabs{Tab:T24} (cf.\ \textbf{D}\ref{tangle:A9}, \textbf{D}\ref{tangle:A17}).
\iN{tangle:A10} $\left(x^3+4x^2+7x+4\right)^{n}+\left(x^2-1\right)(3x+4)^n$, \Tabs{Tab:T27}.
\iR{tangle:A10} $(x+1)\left(x^3+4x^2+7x+4\right)^{n}+\left(x^2-1\right)(3x+4)^n$, \Tabs{Tab:T28}.
\end{description}

\item $ \cla{\dfrac{1}{[2]}+\dfrac{1}{[2]}} = \left\{\dfrac{1}{[2]}+\dfrac{1}{[2]}\right\} $. \label{tangle:A11}
\begin{description}
\iB{tangle:A11} $\left(x^2+4x+4\right)\br{[0]} +(3x+4)\br{[\infty]} $.
\iD{tangle:A11} $x\left(2x+2\right)^{2n} $, \Tabs{Tab:T29} (cf.\ \textbf{D}\ref{tangle:A30}).
\iN{tangle:A11} $\left(2x+2\right)^{2n}+\left(x^2-1\right)\left(x+2\right)^{2n}$, \Tabs{Tab:T30}.
\iR{tangle:A11} $(x+1)\left(2x+2\right)^{2n}+\left(x^2-1\right)\left(x+2\right)^{2n}$, \Tabs{Tab:T31}.
\end{description}

\item $ \cla{[2]*\dfrac{1}{[2]}} = \left\{[2]*\dfrac{1}{[2]},\dfrac{1}{[2]}*[2],[1]*[2]*[1]\right\} $. \label{tangle:A12}
\begin{description}
\iB{tangle:A12} $\left(2x^2+6x+5\right)\br{[0]} +(x+2)\br{[\infty]} $.
\iD{tangle:A12} $x\left(3x^2+8x+5\right)^{n} $, \Tabs{Tab:T32} (cf.\ \textbf{D}\ref{tangle:A14}).
\iN{tangle:A12} $\left(3x^2+8x+5\right)^{n}+\left(x^2-1\right)\left(2x^2+6x+5\right)^n$, \Tabs{Tab:T33}.
\iR{tangle:A12} $(x+1)\left(3x^2+8x+5\right)^{n}+\left(x^2-1\right)\left(2x^2+6x+5\right)^n$, \Tabs{Tab:T34}.
\end{description}

\item $ \cla{[2]+\dfrac{1}{[2]}} = \left\{[2]+\dfrac{1}{[2]},\dfrac{1}{[2]}+[2], [1]+\dfrac{1}{[2]}+[1]\right\} $. \label{tangle:A13}
\begin{description}
\iB{tangle:A13} $ (x+2)\br{[0]} +\left(2x^2+6x+5\right)\br{[\infty]} $.
\iD{tangle:A13} $x\left(2x^3+6x^2+6x+2\right)^{n} $, \Tabs{Tab:T35} (cf.\  \textbf{D}\ref{tangle:A23}, \textbf{D}\ref{tangle:A27}, \textbf{D}\ref{tangle:A29}, \textbf{D}\ref{tangle:A33}).
\iN{tangle:A13} $\left(2x^3+6x^2+6x+2\right)^{n}+\left(x^2-1\right)(x+2)^n$, \Tabs{Tab:T36}.
\iR{tangle:A13} $(x+1)\left(2x^3+6x^2+6x+2\right)^{n}+\left(x^2-1\right)(x+2)^n$, \Tabs{Tab:T37}.
\end{description}

\item $ \begin{aligned}[t]
\cla{[1]*\left([1]+\dfrac{1}{[2]}\right)}  = & \left\{[1]*\left([1]+\dfrac{1}{[2]}\right),\left([1]+\dfrac{1}{[2]}\right)*[1],[1]*\left(\dfrac{1}{[2]}+[1]\right),\right.\\
& \left.\left(\dfrac{1}{[2]}+[1]\right)*[1]\right\}.		
\end{aligned} $\label{tangle:A14}
\begin{description}
\iB{tangle:A14} $\left(x^2+5x+5\right)\br{[0]} +(2x+3)\br{[\infty]} $.
\iD{tangle:A14} $x\left(3x^2+8x+5\right)^{n} $, \Tabs{Tab:T32}  (cf.\ \textbf{D}\ref{tangle:A12}).
\iN{tangle:A14} $\left(3x^2+8x+5\right)^{n}+\left(x^2-1\right)\left(x^2+5x+5\right)^n$, \Tabs{Tab:T38}.
\iR{tangle:A14} $(x+1)\left(3x^2+8x+5\right)^{n}+\left(x^2-1\right)\left(x^2+5x+5\right)^n$, \Tabs{Tab:T39}.
\end{description}

\item $ \clb{[1]+\left([2]*[1]\right)} = \big\{[1]+\left([2]*[1]\right),\left([2]*[1]\right)+[1],[1]+\left([1]*[2]\right),\left([1]*[2]\right)+[1]\big\} $. \label{tangle:A15}
\begin{description}
\iB{tangle:A15} $ (2x+3)\br{[0]} +\left(x^2+5x+5\right)\br{[\infty]} $.
\iD{tangle:A15} $x\left(x^3+5x^2+7x+3\right)^{n} $, \Tabs{Tab:T40} (cf.\ \textbf{D}\ref{tangle:A16}, \textbf{D}\ref{tangle:A22}, \textbf{D}\ref{tangle:A24}, \textbf{D}\ref{tangle:A32}).
\iN{tangle:A15} $\left(x^3+5x^2+7x+3\right)^{n}+\left(x^2-1\right)\left(2x+3\right)^n$, \Tabs{Tab:T41}.
\iR{tangle:A15} $(x+1)\left(x^3+5x^2+7x+3\right)^{n}+\left(x^2-1\right)\left(2x+3\right)^n$, \Tabs{Tab:T42}.
\end{description}

\item $ \cla{[1]+\dfrac{1}{[3]}} = \left\{[1]+\dfrac{1}{[3]}, \dfrac{1}{[3]}+[1]\right\} $. \label{tangle:A16}
\begin{description}
\iB{tangle:A16} $ \left(x^2+3x+3\right)\br{[0]} +\left(x^2+4x+4\right)\br{[\infty]} $.
\iD{tangle:A16} $x\left(x^3+5x^2+7x+3\right)^{n} $, \Tabs{Tab:T40} (cf.\ \textbf{D}\ref{tangle:A15}, \textbf{D}\ref{tangle:A22}, \textbf{D}\ref{tangle:A24}, \textbf{D}\ref{tangle:A32}).
\iN{tangle:A16} $\left(x^3+5x^2+7x+3\right)^{n}+\left(x^2-1\right)\left(x^2+3x+3\right)^n$, \Tabs{Tab:T43}.
\iR{tangle:A16} $(x+1)\left(x^3+5x^2+7x+3\right)^{n}+\left(x^2-1\right)\left(x^2+3x+3\right)^n$, \Tabs{Tab:T44}.
\end{description}

\item $ \clb{[1]*[3]} = \big\{[1]*[3],[3]*[1]\big\} $. \label{tangle:A17}
\begin{description}
\iB{tangle:A17} $ \left(x^2+4x+4\right)\br{[0]} +\left(x^2+3x+3\right)\br{[\infty]}  $.
\iD{tangle:A17} $x\left(x^3+4x^2+7x+4\right)^{n} $, \Tabs{Tab:T24} (cf.\ \textbf{D}\ref{tangle:A9}, \textbf{D}\ref{tangle:A10}).
\iN{tangle:A17} $\left(x^3+4x^2+7x+4\right)^{n}+\left(x^2-1\right)\left(x^2+4x+4\right)^n$, \Tabs{Tab:T45}.
\iR{tangle:A17} $(x+1)\left(x^3+4x^2+7x+4\right)^{n}+\left(x^2-1\right)\left(x^2+4x+4\right)^n$, \Tabs{Tab:T46}.
\end{description}
\end{enumerate}

\begin{notation}
As previously raised in \Rem{Rem:bracketequal2}, the class $ \clb{T\#K} $ is also equal to the set $ \clb{T}\#\clb{K} :=\left\{T\#K:T\in\clb{T}, K\in\clb{K}\right\} $. We identify the following classes for $ \clb{K} $, where $ c(K)\in\{1,2,3,4\} $ \cite[p.\ 14]{Arnold}, \cite{Demaine}.
\begin{itemize}
\item[\textbf{K1}:] $ \cla{\protect\includegraphics[width=.055\linewidth,valign=c]{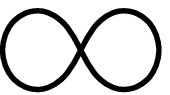}} =\left\{K:\br{K}=x^2 + x\right\}=\left\{\protect\includegraphics[width=.055\linewidth,valign=c]{K1},\protect\includegraphics[width=.06\linewidth,valign=c]{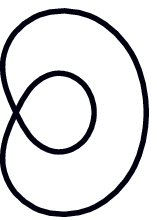}\right\} $;

\item[\textbf{K2}:]  $ \cla{\protect\includegraphics[width=.08\linewidth,valign=c]{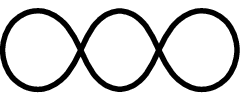}} = \left\{K:\br{K}=x^3 + 2x^2 + x\right\} = \left\{\protect\includegraphics[width=.077\linewidth,valign=c]{K2},\protect\includegraphics[width=.077\linewidth,valign=c]{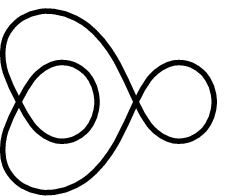},\protect\includegraphics[width=.077\linewidth,valign=c]{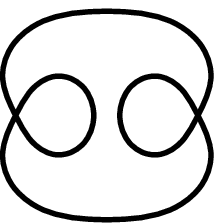},\protect\includegraphics[width=.077\linewidth,valign=c]{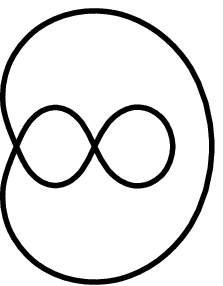},\protect\includegraphics[width=.077\linewidth,valign=c]{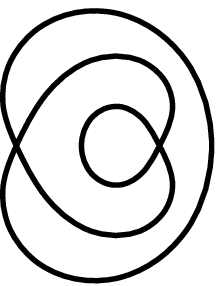}\right\}$;

\item[\textbf{K3}:]  $ \cla{\protect\includegraphics[width=.06\linewidth,valign=c]{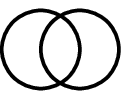}} = \left\{K:\br{K}=2x^2 + 2x\right\} = \left\{\protect\includegraphics[width=.07\linewidth,valign=c]{K3}\right\};$

\item[\textbf{K4}:]  $ \begin{aligned}[t]
\cla{\protect\includegraphics[width=.1\linewidth,valign=c]{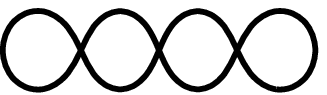}}  = &\left\{K:\br{K}=x^4 + 3x^3 + 3x^2 + x\right\}  = \Biggl\{\protect\includegraphics[width=.1\linewidth,valign=c]{K4},\protect\includegraphics[width=.1\linewidth,valign=c]{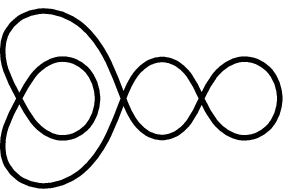},\protect\includegraphics[width=.1\linewidth,valign=c]{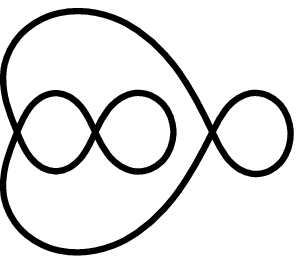},\\
& \protect\includegraphics[width=.085\linewidth,valign=c]{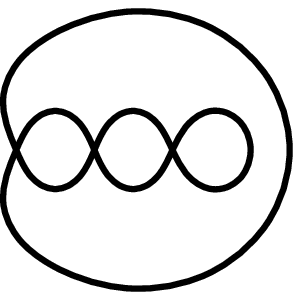},\protect\includegraphics[width=.085\linewidth,valign=c]{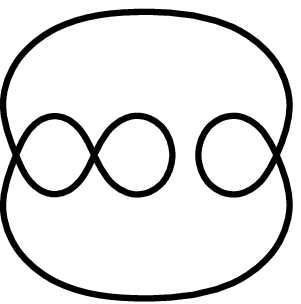},\protect\includegraphics[width=.085\linewidth,valign=c]{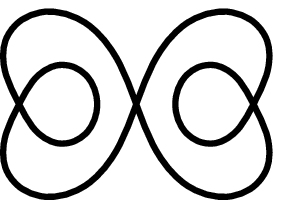},\protect\includegraphics[width=.085\linewidth,valign=c]{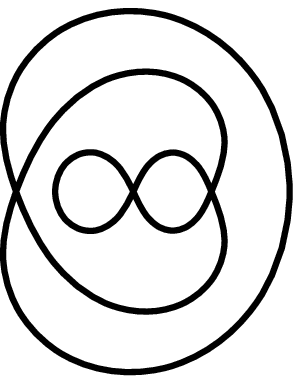},\protect\includegraphics[width=.085\linewidth,valign=c]{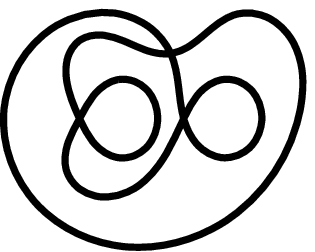},\protect\includegraphics[width=.085\linewidth,valign=c]{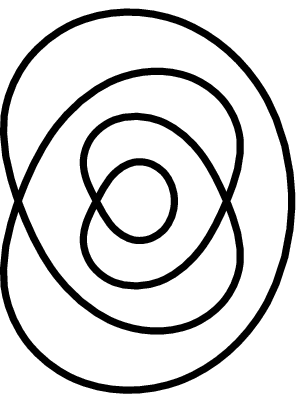},\protect\includegraphics[width=.11\linewidth,valign=c]{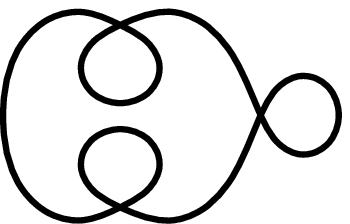},\protect\includegraphics[width=.11\linewidth,valign=c]{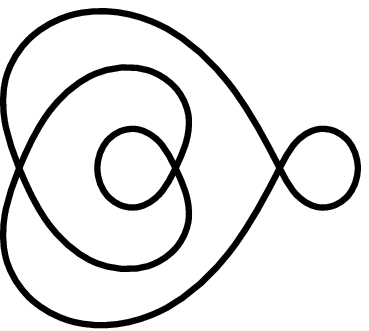},\\
& \protect\includegraphics[width=.09\linewidth,valign=c]{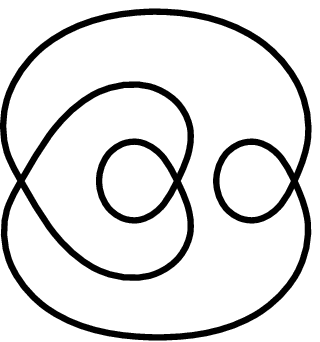},\protect\includegraphics[width=.09\linewidth,valign=c]{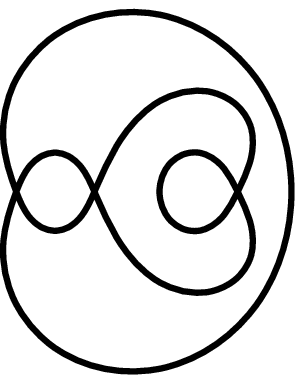}, \protect\includegraphics[width=.08\linewidth,valign=c]{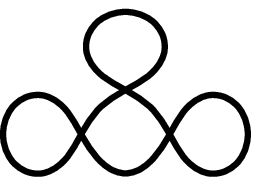}, \protect\includegraphics[width=.125\linewidth,valign=c]{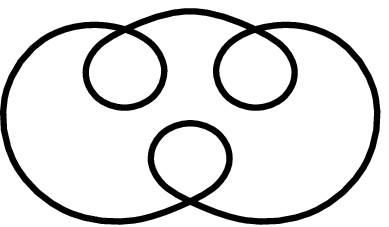},\protect\includegraphics[width=.11\linewidth,valign=c]{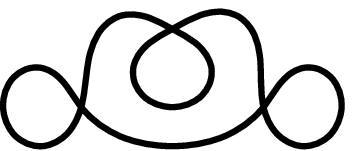},\protect\includegraphics[width=.09\linewidth,valign=c]{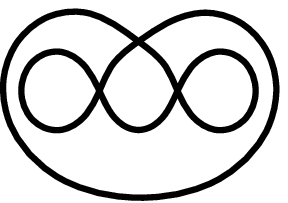},\protect\includegraphics[width=.1\linewidth,valign=c]{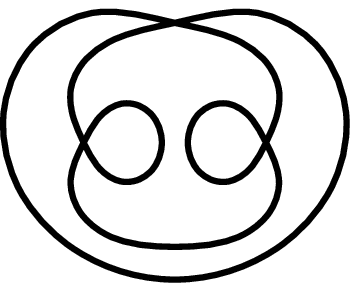}\Biggl\};
\end{aligned} $

\item[\textbf{K5}:]  $ \cla{\protect\includegraphics[width=.05\linewidth,valign=c]{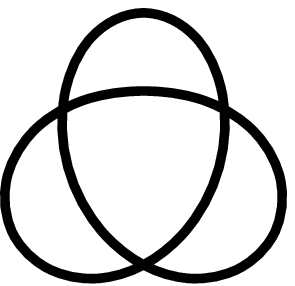}} = \left\{K:\br{K}=x^3 + 4x^2 + 3x\right\}  = \left\{\protect\includegraphics[width=0.07\linewidth,valign=c]{K5},\protect\includegraphics[width=.055\linewidth,valign=c]{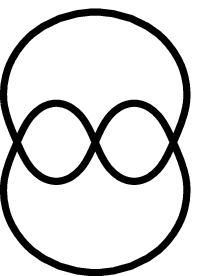}\right\}$;

\item[\textbf{K6}:]  $ \cla{\protect\includegraphics[width=.07\linewidth,valign=c]{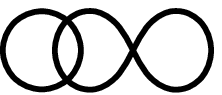}} = \left\{K:\br{K}=2x^3 + 4x^2 + 2x\right\}=\left\{\protect\includegraphics[width=.07\linewidth,valign=c]{K6},\protect\includegraphics[width=.07\linewidth,valign=c]{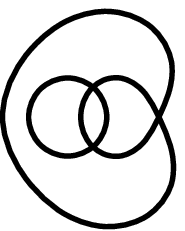},\protect\includegraphics[width=.07\linewidth,valign=c]{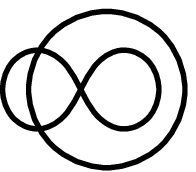},\protect\includegraphics[width=.09\linewidth,valign=c]{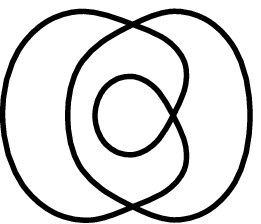},\protect\includegraphics[width=.07\linewidth,valign=c]{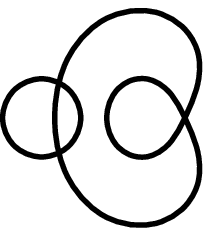}\right\}$.
\end{itemize}
Note also that 
\begin{itemize}
\item  $ \cla{\protect\includegraphics[width=.08\linewidth,valign=c]{K2}} =   \cla{\protect\includegraphics[width=.055\linewidth,valign=c]{K1}}\#  \cla{\protect\includegraphics[width=.055\linewidth,valign=c]{K1}}$;
\item  $ \cla{\protect\includegraphics[width=.1\linewidth,valign=c]{K4}} =   \cla{\protect\includegraphics[width=.08\linewidth,valign=c]{K2}}\#  \cla{\protect\includegraphics[width=.055\linewidth,valign=c]{K1}}$;
\item  $ \cla{\protect\includegraphics[width=.07\linewidth,valign=c]{K6}} =   \cla{\protect\includegraphics[width=.06\linewidth,valign=c]{K3}}\#  \cla{\protect\includegraphics[width=.055\linewidth,valign=c]{K1}}$.
\end{itemize}
\end{notation}

Now we have the following list for locally knotted tangles.	
\begin{enumerate}[resume]
\item $ \clb{[1]}\#\cla{\protect\includegraphics[width=.055\linewidth,valign=c]{K1}} $. \label{tangle:A18}
\begin{description}
\iB{tangle:A18} $ (x+1)\br{[0]} +(x+1)\br{[\infty]}$.
\iD{tangle:A18} $x\left(x+1\right)^{2n} $, \Tabs{Tab:T4} (cf.\ \textbf{D}\ref{tangle:A2}).
\iN{tangle:A18} $\left(x+1\right)^{2n} +\left(x^2-1\right)(x+1)^n$, \Tabs{Tab:T47}.
\iR{tangle:A18} $\left(x+1\right)^{2n+1} +\left(x^2-1\right)(x+1)^n$, \Tabs{Tab:T48}.
\end{description}

\item $ \clb{[2]}\#\cla{\protect\includegraphics[width=.055\linewidth,valign=c]{K1}} $. \label{tangle:A19}  
\begin{description}
\iB{tangle:A19} $ (x+1)\br{[0]} +(x+3x+2)\br{[\infty]}$.
\iD{tangle:A19} $x\left(x+1\right)^{3n} $, \Tabs{Tab:T10} (cf.\ \textbf{D}\ref{tangle:A4}, \textbf{D}\ref{tangle:A25}).
\iN{tangle:A19} $\left(x+1\right)^{3n} +\left(x^2-1\right)(x+1)^n$, \Tabs{Tab:T49}.
\iR{tangle:A19} $\left(x+1\right)^{3n+1} +\left(x^2-1\right)(x+1)^n$, \Tabs{Tab:T50}.
\end{description}

\item $ \cla{\dfrac{1}{[2]}}\#\cla{\protect\includegraphics[width=.055\linewidth,valign=c]{K1}} $. \label{tangle:A20}  
\begin{description}
\iB{tangle:A20} $ (x+3x+2)\br{[0]} +(x+1)\br{[\infty]} $.
\iD{tangle:A20} $ x\left(2x^2+4x+2\right)^{n} $, \Tabs{Tab:T16} (cf.\ \textbf{D}\ref{tangle:A6}, \textbf{D}\ref{tangle:A28}).
\iN{tangle:A20} $ \left(2x^2+4x+2\right)^{n} +\left(x^2-1\right)\left(x^2+3x+2\right)^n$, \Tabs{Tab:T51}.
\iR{tangle:A20} $ (x+1)\left(2x^2+4x+2\right)^{n} +\left(x^2-1\right)\left(x^2+3x+2\right)^n$, \Tabs{Tab:T52}.
\end{description}

\item $ \clb{[3]}\#\cla{\protect\includegraphics[width=.055\linewidth,valign=c]{K1}} $.  \label{tangle:A21} 
\begin{description}
\iB{tangle:A21} $ (x+1)\br{[0]} +\left(x^3+4x^2+6x+3\right)\br{[\infty]}$ 
\iD{tangle:A21} $x\left(x+1\right)^{4n} $, \Tabs{Tab:T21}  (cf.\ \textbf{D}\ref{tangle:A8}, \textbf{D}\ref{tangle:A26}, \textbf{D}\ref{tangle:A31}).
\iN{tangle:A21} $\left(x+1\right)^{4n} +\left(x^2-1\right)(x+1)^n$, \Tabs{Tab:T53}.
\iR{tangle:A21} $\left(x+1\right)^{4n+1} +\left(x^2-1\right)(x+1)^n$, \Tabs{Tab:T54}.
\end{description}

\item $ \cla{\dfrac{1}{[3]}}\#\cla{\protect\includegraphics[width=.055\linewidth,valign=c]{K1}} $.  \label{tangle:A22}
\begin{description}
\iB{tangle:A22} $ \left(x^3+4x^2+6x+3\right)\br{[0]} +(x+1)\br{[\infty]} $.
\iD{tangle:A22} $ x\left(x^3+5x^2+7x+3\right)^{n} $, \Tabs{Tab:T40}   (cf.\ \textbf{D}\ref{tangle:A15}, \textbf{D}\ref{tangle:A16}, \textbf{D}\ref{tangle:A24}, \textbf{D}\ref{tangle:A32}).
\iN{tangle:A22} $\left(x^3+5x^2+7x+3\right)^{n}+\left(x^2-1\right)\left(x^3+4x^2+6x+3\right)^n$, \Tabs{Tab:T55}.
\iR{tangle:A22} $(x+1)\left(x^3+5x^2+7x+3\right)^{n}+\left(x^2-1\right)\left(x^3+4x^2+6x+3\right)^n$, \Tabs{Tab:T56}.
\end{description}

\item $ \cla{[1]+\dfrac{1}{[2]}}\#\cla{\protect\includegraphics[width=.055\linewidth,valign=c]{K1}} $.  \label{tangle:A23}
\begin{description}
\iB{tangle:A23} $ \left(x^2+3x+2\right)\br{[0]} +\left(2x^2+5x+3\right)\br{[\infty]}$.
\iD{tangle:A23} $ x\left(2x^3+6x^2+6x+2\right)^{n} $, \Tabs{Tab:T35} (cf.\ \textbf{D}\ref{tangle:A13}, \textbf{D}\ref{tangle:A27}, \textbf{D}\ref{tangle:A29}, \textbf{D}\ref{tangle:A33}).
\iN{tangle:A23} $ \left(2x^3+6x^2+6x+2\right)^{n}+\left(x^2-1\right)\left(x^2+3x+2\right)^n$, \Tabs{Tab:T57}.
\iR{tangle:A23} $ (x+1)\left(2x^3+6x^2+6x+2\right)^{n}+\left(x^2-1\right)\left(x^2+3x+2\right)^n$, \Tabs{Tab:T58}.
\end{description}

\item $ \pi\big([1]*[2]\big)\#\cla{\protect\includegraphics[width=.055\linewidth,valign=c]{K1}} $.  \label{tangle:A24}
\begin{description}
\iB{tangle:A24} $ \left(2x^2+5x+3\right)\br{[0]} +\left(x^2+3x+2\right)\br{[\infty]}$.
\iD{tangle:A24} $ x\left(x^3+5x^2+7x+3\right)^{n} $, \Tabs{Tab:T40} (cf.\ \textbf{D}\ref{tangle:A15}, \textbf{D}\ref{tangle:A16}, \textbf{D}\ref{tangle:A22}, \textbf{D}\ref{tangle:A32}).
\iN{tangle:A24} $ \left(x^3+5x^2+7x+3\right)^{n}+\left(x^2-1\right)\left(2x^2+5x+3\right)^n$, \Tabs{Tab:T59}.
\iR{tangle:A24} $ (x+1)\left(x^3+5x^2+7x+3\right)^{n}+\left(x^2-1\right)\left(2x^2+5x+3\right)^n$, \Tabs{Tab:T60}.
\end{description}

\item $\pi\big([1]\big)\#\cla{\protect\includegraphics[width=.08\linewidth,valign=c]{K2}} $.  \label{tangle:A25} 
\begin{description}
\iB{tangle:A25} $ \left(x^2+2x+1\right)\br{[0]} +\left(x^2+2x+1\right)\br{[\infty]}$.
\iD{tangle:A25} $ x\left(x+1\right)^{3n} $, \Tabs{Tab:T10} (cf.\ \textbf{D}\ref{tangle:A4}, \textbf{D}\ref{tangle:A19}).
\iN{tangle:A25} $ \left(x+1\right)^{3n} +\left(x^2-1\right)\left(x+1\right)^{2n}$, \Tabs{Tab:T61}.
\iR{tangle:A25} $ \left(x+1\right)^{3n+1} +\left(x^2-1\right)\left(x+1\right)^{2n}$, \Tabs{Tab:T62}.
\end{description}

\item $\pi\big([2]\big)\#\cla{\protect\includegraphics[width=.08\linewidth,valign=c]{K2}} $. \label{tangle:A26} 
\begin{description}
\iB{tangle:A26} $ \left(x^2+2x+1\right)\br{[0]} +\left(x^3+4x^2+5x+2\right)\br{[\infty]}$.
\iD{tangle:A26} $ x\left(x+1\right)^{4n} $, \Tabs{Tab:T21}   (cf.\ \textbf{D}\ref{tangle:A8}, \textbf{D}\ref{tangle:A21}, \textbf{D}\ref{tangle:A31}).
\iN{tangle:A26} $ \left(x+1\right)^{4n} +\left(x^2-1\right)\left(x+1\right)^{2n}$, \Tabs{Tab:T63}.
\iR{tangle:A26} $ \left(x+1\right)^{4n+1} +\left(x^2-1\right)\left(x+1\right)^{2n}$, \Tabs{Tab:T64}.
\end{description}

\item $\cla{\dfrac{1}{[2]}}\#\cla{\protect\includegraphics[width=.08\linewidth,valign=c]{K2}} $.  \label{tangle:A27} 
\begin{description}
\iB{tangle:A27} $ \left(x^3+4x^2+5x+2\right)\br{[0]} +\left(x^2+2x+1\right)\br{[\infty]} $.
\iD{tangle:A27} $ x\left(2x^3+6x^2+6x+2\right)^{n} $, \Tabs{Tab:T35} (cf.\ \textbf{D}\ref{tangle:A13}, \textbf{D}\ref{tangle:A23}, \textbf{D}\ref{tangle:A29}, \textbf{D}\ref{tangle:A33}).
\iN{tangle:A27} $ \left(2x^3+6x^2+6x+2\right)^{n} +\left(x^2-1\right)\left(x^3+4x^2+5x+2\right)^n$, \Tabs{Tab:T65}.
\iR{tangle:A27} $ (x+1)\left(2x^3+6x^2+6x+2\right)^{n} +\left(x^2-1\right)\left(x^3+4x^2+5x+2\right)^n$, \Tabs{Tab:T66}.
\end{description}

\item $\pi\big([1]\big)\#\cla{\protect\includegraphics[width=.06\linewidth,valign=c]{K3}} $. \label{tangle:A28}
\begin{description}
\iB{tangle:A28} $ (2x+2)\br{[0]} +(2x+2)\br{[\infty]}$.
\iD{tangle:A28} $x\left(2x^2+4x+2\right)^n $, \Tabs{Tab:T16} (cf.\ \textbf{D}\ref{tangle:A6}, \textbf{D}\ref{tangle:A20}).
\iN{tangle:A28} $\left(2x^2+4x+2\right)^n +\left(x^2-1\right)(2x+2)^n$, \Tabs{Tab:T67}.
\iR{tangle:A28} $(x+1)\left(2x^2+4x+2\right)^n +\left(x^2-1\right)(2x+2)^n$,	 \Tabs{Tab:T68}.
\end{description}

\item $\pi\big([2]\big)\#\cla{\protect\includegraphics[width=.06\linewidth,valign=c]{K3}} $. \label{tangle:A29} 
\begin{description}
\iB{tangle:A29} $ (2x+2)\br{[0]} +\left(2x^2+6x+4\right)\br{[\infty]}$.
\iD{tangle:A29} $ x\left(2x^3+6x^2+6x+2\right)^{n} $, \Tabs{Tab:T35} (cf.\ \textbf{D}\ref{tangle:A13}, \textbf{D}\ref{tangle:A23}, \textbf{D}\ref{tangle:A27}, \textbf{D}\ref{tangle:A33}).
\iN{tangle:A29} $ \left(2x^3+6x^2+6x+2\right)^{n} +\left(x^2-1\right)\left(2x+2\right)^n$, \Tabs{Tab:T69}.
\iR{tangle:A29} $ (x+1)\left(2x^3+6x^2+6x+2\right)^{n} +\left(x^2-1\right)\left(2x+2\right)^n$, \Tabs{Tab:T70}.
\end{description}

\item  $ \cla{\dfrac{1}{[2]}}\#\cla{\protect\includegraphics[width=.06\linewidth,valign=c]{K3}} $. \label{tangle:A30} 
\begin{description}
\iB{tangle:A30} $ \left(2x^2+6x+4\right)\br{[0]} +(2x+2)\br{[\infty]}$.
\iD{tangle:A30} $ x\left(2x+2\right)^{2n} $, \Tabs{Tab:T29} (cf.\ \textbf{D}\ref{tangle:A11}).
\iN{tangle:A30} $ \left(2x+2\right)^{2n} +\left(x^2-1\right)\left(2x^2+6x+4\right)^n$, \Tabs{Tab:T71}.
\iR{tangle:A30} $ (x+1)\left(2x+2\right)^{2n} +\left(x^2-1\right)\left(2x^2+6x+4\right)^n$, \Tabs{Tab:T72}.
\end{description}

\item $\pi\big([1]\big)\#\cla{\protect\includegraphics[width=.1\linewidth,valign=c]{K4}} $. \label{tangle:A31} 
\begin{description}
\iB{tangle:A31} $ \left(x^3+3x^2+3x+1\right)\br{[0]} +\left(x^3+3x^2+3x+1\right)\br{[\infty]}$.
\iD{tangle:A31} $ x\left(x+1\right)^{4n} $,  \Tabs{Tab:T21}  (cf.\ \textbf{D}\ref{tangle:A8}, \textbf{D}\ref{tangle:A21}, \textbf{D}\ref{tangle:A26}).
\iN{tangle:A31} $ \left(x+1\right)^{4n} +\left(x^2-1\right)\left(x+1\right)^{3n}$, \Tabs{Tab:T73}.
\iR{tangle:A31} $ \left(x+1\right)^{4n+1} +\left(x^2-1\right)\left(x+1\right)^{3n}$, \Tabs{Tab:T74}.
\end{description}

\item  $\pi\big([1]\big)\#\cla{\protect\includegraphics[width=.05\linewidth,valign=c]{K5}} $. \label{tangle:A32} 
\begin{description}
\iB{tangle:A32} $ \left(x^2+4x+3\right)\br{[0]} +\left(x^2+4x+3\right)\br{[\infty]}$.
\iD{tangle:A32} $ x\left(x^3+5x^2+7x+3\right)^{n} $, \Tabs{Tab:T40} (cf.\ \textbf{D}\ref{tangle:A15}, \textbf{D}\ref{tangle:A16}, \textbf{D}\ref{tangle:A22}, \textbf{D}\ref{tangle:A24}).
\iN{tangle:A32} $ \left(x^3+5x^2+7x+3\right)^{n}+\left(x^2-1\right)\left(x^2+4x+3\right)^n$, \Tabs{Tab:T75}.
\iR{tangle:A32} $ (x+1)\left(x^3+5x^2+7x+3\right)^{n}+\left(x^2-1\right)\left(x^2+4x+3\right)^n$, \Tabs{Tab:T76}.
\end{description}

\item $\pi\big([1]\big)\#\cla{\protect\includegraphics[width=.07\linewidth,valign=c]{K6}} $. \label{tangle:A33} 
\begin{description}
\iB{tangle:A33} $ \left(2x^2+4x+2\right)\br{[0]} +\left(2x^2+4x+2\right)\br{[\infty]}$.
\iD{tangle:A33} $ x\left(2x^3+6x^2+6x+2\right)^{n} $, \Tabs{Tab:T35} (cf.\ \textbf{D}\ref{tangle:A13}, \textbf{D}\ref{tangle:A23}, \textbf{D}\ref{tangle:A27}, \textbf{D}\ref{tangle:A29}).
\iN{tangle:A33} $ \left(2x^3+6x^2+6x+2\right)^{n} +\left(x^2-1\right)\left(2x^2+4x+2\right)^n$, \Tabs{Tab:T77}.
\iR{tangle:A33} $ (x+2)\left(2x^3+6x^2+6x+2\right)^{n} +\left(x^2-1\right)\left(2x^2+4x+2\right)^n$, \Tabs{Tab:T78}.
\end{description}
\end{enumerate}

\begin{remark}
Except for the classes of $ [0] $ and $ [\infty] $ below, we do not take into consideration other tangles verifying $ a(A).b(A)=0 $.
\begin{enumerate}[resume]
\item $ \clb{[0]}=\big\{[0]\big\} $.\label{tangle:A34}
\begin{description}
\iB{tangle:A34} $\br{[0]} $.
\iD{tangle:A34} $ x $.
\iN{tangle:A34} $ x^2$.
\iR{tangle:A34} $ x^2+x$.
\end{description}

\item $ \clb{[0]}=\big\{[\infty]\big\} $.\label{tangle:A35}
\begin{description}
\iB{tangle:A35} $\br{[\infty]} $.
\iD{tangle:A35} $ x^{n+1} $, \Tabs{Tab:T79}.
\iN{tangle:A35} $ x^2 $ if $ n=0 $, $ x^n$ otherwise, \Tabs{Tab:T80}.
\iR{tangle:A35} $ x^2 + x $ if $ n=0 $, $ x^{n+1}+x^n$ otherwise, \Tabs{Tab:T81}.
\end{description}
\end{enumerate}
\end{remark}
\subsection{Tables of coefficients}\label{subsec:coefficients}
Tables that are listed here consist of the coefficients in the expansion of the bracket polynomials seen in \Subsec{subsec:presentation} for small $ n $.

\begin{table}[H]
\centering
$
\begin{array}{@{}c|lrrrrrrr@{}}%
n \backslash k & 0 & 1 & 2 & 3 & 4 & 5 & 6\\
\midrule
0 & 0 & 1\\
1 & 0 & 1 & 1\\
2 & 0 & 1 & 2 & 1\\
3 & 0 & 1 & 3 & 3 & 1\\
4 & 0 & 1 & 4 & 6 & 4 & 1\\
5 & 0 & 1 & 5 & 10 & 10 & 5 & 1\\
\end{array}$
\caption{$ \cof{x^k}x(x+1)^n $ (cf.\ \textbf{D}\ref{tangle:A1}) \cite[\seqnum{A007318}]{Sloane}.}
\label{Tab:T1}
\end{table}

\begin{table}[H]
\centering
$\begin{array}{@{}c|rrrrrr@{}}%
n \backslash k & 0 & 1 & 2 & 3 & 4 & 5 \\
\midrule
0 & 0 & 0 & 1\\
1 & 0 & 1 & 1\\
2 & 0 & 2 & 2\\
3 & 0 & 3 & 4 & 1\\
4 & 0 & 4 & 7 & 4 & 1\\
5 & 0 & 5 & 11 & 10 & 5 & 1\\
\end{array}$
\caption{$ \cof{x^k}\left((x+1)^n+x^2-1\right) $  (cf.\ \textbf{N}\ref{tangle:A1}) \cite[\seqnum{A300453}]{Sloane}.}
\label{Tab:T2}
\end{table}

\begin{table}[H] 
\centering 
$\begin{array}{@{}c|rrrrrrr@{}}%
n \backslash k &0 & 1 & 2 & 3 & 4 & 5 & 6 \\ 
\midrule 
0 & 0 & 1 & 1 \\ 
1 & 0 & 2 & 2 \\ 
2 & 0 & 3 & 4 & 1 \\ 
3 & 0 & 4 & 7 & 4 & 1 \\ 
4 & 0 & 5 & 11 & 10 & 5 & 1 \\ 
5 & 0 & 6 & 16 & 20 & 15 & 6 & 1 \\
\end{array}$
\caption{$ \cof{x^k}\left((x+1)^{n+1}+x^2-1\right) $  (cf.\ \textbf{R}\ref{tangle:A1}) \cite[\seqnum{A300453}]{Sloane}.}
\label{Tab:T3}
\end{table}

\begin{table}[H]
\centering
$\begin{array}{@{}c|rrrrrrrrrrrr@{}}%
n \backslash k & 0 & 1 & 2 & 3 & 4 & 5 & 6 & 7 & 8 & 9 & 10 & 11\\
\midrule
0 & 0 & 1\\
1 & 0 & 1 & 2 & 1\\
2 & 0 & 1 & 4 & 6 & 4 & 1\\
3 & 0 & 1 & 6 & 15 & 20 & 15 & 6 & 1\\
4 & 0 & 1 & 8 & 28 & 56 & 70 & 56 & 28 & 8 & 1\\
5 & 0 & 1 & 10 & 45 & 120 & 210 & 252 & 210 & 120 & 45 & 10 & 1\\
\end{array}$
\caption{$ \cof{x^k}\left(x(x+1)^{2n}\right) $ (cf.\ \textbf{D}\ref{tangle:A2}, \textbf{D}\ref{tangle:A18})  \cite[\seqnum{A034870}]{Sloane}.}	
\label{Tab:T4}
\end{table}

\begin{table}[H]
\centering
$\begin{array}{@{}c|rrrrrrrrrrr@{}}%
n \backslash k & 0 & 1 & 2 & 3 & 4 & 5 & 6 & 7 & 8 & 9 & 10\\
\midrule
0 & 0 & 0 & 1\\
1 & 0 & 2 & 2\\
2 & 0 & 4 & 7 & 4 & 1\\
3 & 0 & 6 & 16 & 20 & 15 & 6 & 1\\
4 & 0 & 8 & 29 & 56 & 70 & 56 & 28 & 8 & 1\\
5 & 0 & 10 & 46 & 120 & 210 & 252 & 210 & 120 & 45 & 10 & 1\\
\end{array}$
\caption{$ \cof{x^k}\left((x+1)^{2n}+x^2-1\right) $  (cf.\ \textbf{N}\ref{tangle:A2}).}
\label{Tab:T5}
\end{table}

\begin{table}[H] 
\centering 
$\begin{array}{@{}c|rrrrrrrrrrrr@{}}%
n \backslash k &0 & 1 & 2 & 3 & 4 & 5 & 6 & 7 & 8 & 9 & 10 & 11 \\ 
\midrule 
0 & 0 & 1 & 1 \\ 
1 & 0 & 3 & 4 & 1 \\ 
2 & 0 & 5 & 11 & 10 & 5 & 1 \\ 
3 & 0 & 7 & 22 & 35 & 35 & 21 & 7 & 1 \\ 
4 & 0 & 9 & 37 & 84 & 126 & 126 & 84 & 36 & 9 & 1 \\ 
5 & 0 & 11 & 56 & 165 & 330 & 462 & 462 & 330 & 165 & 55 & 11 & 1\\  
\end{array}$ 
\caption{$ \cof{x^k}\left((x+1)^{2n+1}+x^2-1\right) $  (cf.\ \textbf{R}\ref{tangle:A2}).}
\label{Tab:T6}
\end{table}

\begin{table}[H]
\centering
$\begin{array}{@{}c|rrrrrrr@{}}%
n \backslash k & 0 & 1 & 2 & 3 & 4 & 5 & 6\\
\midrule
0 & 0 & 1\\
1 & 0 & 2 & 2\\
2 & 0 & 4 & 8 & 4\\
3 & 0 & 8 & 24 & 24 & 8\\
4 & 0 & 16 & 64 & 96 & 64 & 16\\
4 & 0 & 32 & 160 & 320 & 320 & 160 & 32\\
\end{array}$
\caption{$ \cof{x^k}\left(x(2x+2)^n\right) $ (cf.\ \textbf{D}\ref{tangle:A3}) \cite[\seqnum{A038208}]{Sloane}.}
\label{Tab:T7}
\end{table}

\begin{table}[H]
\centering
$\begin{array}{@{}c|rrrrrrrr@{}}%
n \backslash k & 0 & 1 & 2 & 3 & 4 & 5 & 6 & 7\\
\midrule
0 & 0 & 0 & 1\\
1 & 0 & 1 & 2 & 1\\
2 & 0 & 4 & 7 & 4 & 1\\
3 & 0 & 12 & 26 & 19 & 6 & 1\\
4 & 0 & 32 & 88 & 88 & 39 & 8 & 1\\
5 & 0 & 80 & 272 & 360 & 230 & 71 & 10 & 1\\
\end{array}$
\caption{$ \cof{x^k}\left((2x+2)^{n} +\left(x^2-1\right)(x+2)^n\right) $ (cf.\ \textbf{N}\ref{tangle:A3}) \cite[\seqnum{A300184}]{Sloane}.}
\label{Tab:T8}
\end{table}

\begin{table}[H] 
\centering 
$\begin{array}{@{}c|rrrrrrrr@{}}%
n \backslash k &0 & 1 & 2 & 3 & 4 & 5 & 6 & 7 \\ 
\midrule 
0 & 0 & 1 & 1 \\ 
1 & 0 & 3 & 4 & 1 \\ 
2 & 0 & 8 & 15 & 8 & 1 \\ 
3 & 0 & 20 & 50 & 43 & 14 & 1 \\ 
4 & 0 & 48 & 152 & 184 & 103 & 24 & 1 \\ 
5 & 0 & 112 & 432 & 680 & 550 & 231 & 42 & 1\\ 
\end{array}$ 
\caption{$ \cof{x^k}\left((x+1)(2x+2)^{n} +\left(x^2-1\right)(x+2)^n\right) $ (cf.\ \textbf{R}\ref{tangle:A3}).}
\label{Tab:T9}
\end{table}

\begin{table}[H]
\centering
$\begin{array}{@{}c|rrrrrrrrrrrrrrr@{}}%
n \backslash k & 0 & 1 & 2 & 3 & 4 & 5 & 6 & 7 & 8 & 9 & 10 & 11 & 12 & 13 & 14\\
\midrule
0 & 0 & 1\\
1 & 0 & 1 & 3 & 3 & 1\\
2 & 0 & 1 & 6 & 15 & 20 & 15 & 6 & 1\\
3 & 0 & 1 & 9 & 36 & 84 & 126 & 126 & 84 & 36 & 9 & 1\\
4 & 0 & 1 & 12 & 66 & 220 & 495 & 792 & 924 & 792 & 495 & 220 & 66 & 12 & 1\\
5 & 0 & 1 & 15 & 105 & 455 & 1365 & 3003 & 5005 & 6435 & 6435 & 5005 & 3003 & 1365 & 455 & \ldots \\% 105 & 15 & 1
\end{array}$
\caption{$ \cof{x^k}\left(x(x+1)^{3n}\right) $  (cf.\ \textbf{D}\ref{tangle:A4}, \textbf{D}\ref{tangle:A25}, \textbf{D}\ref{tangle:A19}).}
\label{Tab:T10}
\end{table}

\begin{table}[H]
\centering
$\begin{array}{@{}c|rrrrrrrrrrrrrr@{}}%
n \backslash k & 0 & 1 & 2 & 3 & 4 & 5 & 6 & 7 & 8 & 9 & 10 & 11 & 12 & 13 \\
\midrule
0 & 0 & 0 & 1\\
1 & 0 & 3 & 4 & 1\\
2 & 0 & 6 & 16 & 20 & 15 & 6 & 1\\
3 & 0 & 9 & 37 & 84 & 126 & 126 & 84 & 36 & 9 & 1\\
4 & 0 & 12 & 67 & 220 & 495 & 792 & 924 & 792 & 495 & 220 & 66 & 12 & 1\\
5 & 0 & 15 & 106 & 455 & 1365 & 3003 & 5005 & 6435 & 6435 & 5005 & 3003 & 1365 & 455 & \ldots\\ %105 & 15 & 1\\
\end{array}$
\caption{$ \cof{x^k}\left((x+1)^{3n}+x^2-1\right) $ (cf.\ \textbf{N}\ref{tangle:A4}).}
\label{Tab:T11}	
\end{table}

\begin{table}[H] 
\centering 
$\begin{array}{@{}c|rrrrrrrrrrrrrr@{}}%
n \backslash k &0 & 1 & 2 & 3 & 4 & 5 & 6 & 7 & 8 & 9 & 10 & 11 & 12  & 13\\ 
\midrule 
0 & 0 & 1 & 1 \\ 
1 & 0 & 4 & 7 & 4 & 1 \\ 
2 & 0 & 7 & 22 & 35 & 35 & 21 & 7 & 1 \\ 
3 & 0 & 10 & 46 & 120 & 210 & 252 & 210 & 120 & 45 & 10 & 1 \\ 
4 & 0 & 13 & 79 & 286 & 715 & 1287 & 1716 & 1716 & 1287 & 715 & 286 & 78 & 13 & \ldots \\ % & 1 \\ 
5 & 0 & 16 & 121 & 560 & 1820 & 4368 & 8008 & 11440 & 12870 & 11440 & 8008 & 4368 & 1820 & \ldots \\ % & 560 & 120 & 16 & 1 \\ 
\end{array}$ 
\caption{$ \cof{x^k}\left((x+1)^{3n+1}+x^2-1\right) $ (cf.\ \textbf{R}\ref{tangle:A4}).}
\label{Tab:T12}	
\end{table}

\begin{table}[H]
\centering
$\begin{array}{@{}c|rrrrrrrrrrrr@{}}%
n \backslash k & 0 & 1 & 2 & 3 & 4 & 5 & 6 & 7 & 8 & 9 & 10 & 11\\
\midrule
0 & 0 & 1\\
1 & 0 & 3 & 4 & 1\\
2 & 0 & 9 & 24 & 22 & 8 & 1\\
3 & 0 & 27 & 108 & 171 & 136 & 57 & 12 & 1\\
4 & 0 & 81 & 432 & 972 & 1200 & 886 & 400 & 108 & 16 & 1\\
5 & 0 & 243 & 1620 & 4725 & 7920 & 8430 & 5944 & 2810 & 880 & 175 & 20 & 1\\
\end{array}$
\caption{$ \cof{x^k}\left(x\left(x^2+4x+3\right)^{n}\right) $ (cf.\ \textbf{D}\ref{tangle:A5}, \textbf{D}\ref{tangle:A7}) \cite[\seqnum{A299989}]{Sloane}.}
\label{Tab:T13}
\end{table}

\begin{table}[H]
\centering
$\begin{array}{@{}c|rrrrrrrrrrrrr@{}}%
n \backslash k  & 0 & 1 & 2 & 3 & 4 & 5 & 6 & 7 & 8 & 9 & 10 & 11 & 12\\
\midrule
0 & 0 & 0 & 1\\
1 & 0 & 1 & 3 & 3 & 1\\
2 & 0 & 6 & 16 & 20 & 15 & 6 & 1\\
3 & 0 & 27 & 90 & 136 & 129 & 84 & 36 & 9 & 1\\
4 & 0 & 108 & 459 & 876 & 1021 & 832 & 501 & 220 & 66 & 12 & 1\\
5 & 0 & 405 & 2133 & 5085 & 7350 & 7321 & 5420 & 3103 & 1375 & 455 & 105 & 15 & 1\\
\end{array}$
\caption{$ \cof{x^k}\left(\left(x^2+4x+3\right)^{n}+\left(x^2-1\right)\left(x^2+3x+3\right)^n\right) $ (cf.\ \textbf{N}\ref{tangle:A5}).}
\label{Tab:T14}	
\end{table}

\begin{table}[H] 
\centering 
$\begin{array}{@{}c|rrrrrrrrrrrrr@{}}%
n \backslash k & 0 & 1 & 2 & 3 & 4 & 5 & 6 & 7 & 8 & 9 & 10 & 11 & 12\\ 
\midrule 
0 & 0 & 1 & 1\\ 
1 & 0 & 4 & 7 & 4 & 1\\ 
2 & 0 & 15 & 40 & 42 & 23 & 7 & 1\\ 
3 & 0 & 54 & 198 & 307 & 265 & 141 & 48 & 10 & 1\\ 
4 & 0 & 189 & 891 & 1848 & 2221 & 1718 & 901 & 328 & 82 & 13 & 1\\ 
5 & 0 & 648 & 3753 & 9810 & 15270 & 15751 & 11364 & 5913 & 2255 & 630 & 125 & 16 & 1\\ 
\end{array}$ 
\caption{$ \cof{x^k}\left((x+1)\left(x^2+4x+3\right)^{n}+\left(x^2-1\right)\left(x^2+3x+3\right)^n\right) $ (cf.\ \textbf{R}\ref{tangle:A5}).}
\label{Tab:T15}
\end{table}

\begin{table}[H]
\centering
$\begin{array}{@{}c|rrrrrrrrrrrr@{}}%
n \backslash k & 0 & 1 & 2 & 3 & 4 & 5 & 6 & 7 & 8 & 9 & 10 & 11\\
\midrule
0 & 0 & 1\\
1 & 0 & 2 & 4 & 2\\
2 & 0 & 4 & 16 & 24 & 16 & 4\\
3 & 0 & 8 & 48 & 120 & 160 & 120 & 48 & 8\\
4 & 0 & 16 & 128 & 448 & 896 & 1120 & 896 & 448 & 128 & 16\\
5 & 0 & 32 & 320 & 1440 & 3840 & 6720 & 8064 & 6720 & 3840 & 1440 & 320 & 32\\
\end{array}$
\caption{$ \cof{x^k}\left(x\left(2x^2+4x+2\right)^{n} \right) $ (cf.\ \textbf{D}\ref{tangle:A6}, \textbf{D}\ref{tangle:A28}, \textbf{D}\ref{tangle:A20})  \cite[\seqnum{A139548}]{Sloane}.}
\label{Tab:T16}	
\end{table}

\begin{table}[H]
\centering
$\begin{array}{@{}c|rrrrrrrrrrr@{}}%
n \backslash k & 0 & 1 & 2 & 3 & 4 & 5 & 6 & 7 & 8 & 9 & 10\\
\midrule
0 & 0 & 0 & 1\\
1 & 0 & 3 & 4 & 1\\
2 & 0 & 12 & 27 & 20 & 5\\
3 & 0 & 36 & 122 & 171 & 126 & 49 & 8\\
4 & 0 & 96 & 440 & 920 & 1143 & 904 & 449 & 128 & 16\\
5 & 0 & 240 & 1392 & 3880 & 6790 & 8103 & 6730 & 3841 & 1440 & 320 & 32\\
\end{array}$
\caption{$ \cof{x^k}\left(\left(2x^2+4x+2\right)^{n}+\left(x^2-1\right)(x+2)^n\right) $ (cf.\ \textbf{N}\ref{tangle:A6}).}
\label{Tab:T17}	
\end{table}

\begin{table}[H] 
\centering 
$\begin{array}{@{}c|rrrrrrrrrrrr@{}}%
n \backslash k &0 & 1 & 2 & 3 & 4 & 5 & 6 & 7 & 8 & 9 & 10 & 11 \\ 
\midrule 
0 & 0 & 1 & 1 \\ 
1 & 0 & 5 & 8 & 3 \\ 
2 & 0 & 16 & 43 & 44 & 21 & 4 \\ 
3 & 0 & 44 & 170 & 291 & 286 & 169 & 56 & 8 \\ 
4 & 0 & 112 & 568 & 1368 & 2039 & 2024 & 1345 & 576 & 144 & 16 \\ 
5 & 0 & 272 & 1712 & 5320 & 10630 & 14823 & 14794 & 10561 & 5280 & 1760 & 352 & 32 \\ 
\end{array}$
\caption{$ \cof{x^k}\left((x+1)\left(2x^2+4x+2\right)^{n}+\left(x^2-1\right)(x+2)^n\right) $ (cf.\ \textbf{R}\ref{tangle:A6}).}
\label{Tab:T18}		 
\end{table}

\begin{table}[H]
\centering
$\begin{array}{@{}c|rrrrrrrrrrr@{}}%
n \backslash k & 0 & 1 & 2 & 3 & 4 & 5 & 6 & 7 & 8 & 9 & 10\\
\midrule
0 & 0 & 0 & 1\\
1 & 0 & 2 & 4 & 2\\
2 & 0 & 12 & 27 & 20 & 5\\
3 & 0 & 54 & 162 & 182 & 93 & 20 & 1\\
4 & 0 & 216 & 837 & 1320 & 1086 & 496 & 124 & 16 & 1\\
5 & 0 & 810 & 3888 & 8010 & 9270 & 6632 & 3050 & 912 & 175 & 20 & 1\\
\end{array}$
\caption{$ \cof{x^k}\left(\left(x^2+4x+3\right)^{n}+\left(x^2-1\right)(2x+3)^n\right) $ (cf.\ \textbf{N}\ref{tangle:A7}).}
\label{Tab:T19}
\end{table}

\begin{table}[H] 
\centering 
$\begin{array}{@{}c|rrrrrrrrrrrr@{}}%
n \backslash k &0 & 1 & 2 & 3 & 4 & 5 & 6 & 7 & 8 & 9 & 10 & 11 \\ 
\midrule 
0 & 0 & 1 & 1 \\ 
1 & 0 & 5 & 8 & 3 \\ 
2 & 0 & 21 & 51 & 42 & 13 & 1 \\ 
3 & 0 & 81 & 270 & 353 & 229 & 77 & 13 & 1 \\ 
4 & 0 & 297 & 1269 & 2292 & 2286 & 1382 & 524 & 124 & 17 & 1 \\ 
5 & 0 & 1053 & 5508 & 12735 & 17190 & 15062 & 8994 & 3722 & 1055 & 195 & 21 & 1 \\ 
\end{array}$ 
\caption{$ \cof{x^k}\left((x+1)\left(x^2+4x+3\right)^{n}+\left(x^2-1\right)(2x+3)^n\right) $  (cf.\ \textbf{R}\ref{tangle:A7}).}
\label{Tab:T20}		
\end{table}

\begin{table}[H]
\centering
$\begin{array}{@{}c|rrrrrrrrrrrrr@{}}%
n \backslash k & 0 & 1 & 2 & 3 & 4 & 5 & 6 & 7 & 8 & 9 & 10 & 11 & 12 \\
\midrule
0 & 0 & 1 \\ 
1 & 0 & 1 & 4 & 6 & 4 & 1 \\
2 & 0 & 1 & 8 & 28 & 56 & 70 & 56 & 28 & 8 & 1 \\
3 & 0 & 1 & 12 & 66 & 220 & 495 & 792 & 924 & 792 & 495 & 220 & 66 & \ldots \\ % & 12 & 1 \\
4 & 0 & 1 & 16 & 120 & 560 & 1820 & 4368 & 8008 & 11440 & 12870 & 11440 & 8008  & \ldots \\ %  & 4368 & 1820 & 560 & 120 & 16 & 1 \\
5 & 0 & 1 & 20 & 190 & 1140 & 4845 & 15504 & 38760 & 77520 & 125970 & 167960 & 184756  & \ldots \\ % & 167960 & 125970 & 77520 & 38760 & 15504 & 4845 & 1140 & 190 & 20 & 1
\end{array}$
\caption{$ \cof{x^k}\left(x(x+1)^{4n}\right) $   (cf.\ \textbf{D}\ref{tangle:A8},  \textbf{D}\ref{tangle:A21},  \textbf{D}\ref{tangle:A26},  \textbf{D}\ref{tangle:A31}).}
\label{Tab:T21}	
\end{table}

\begin{table}[H]
\centering
%	\resizebox{\linewidth}{!}{%
$\begin{array}{@{}c|rrrrrrrrrrrr@{}}%
n\backslash k &0 & 1 & 2 & 3 & 4 & 5 & 6 & 7 & 8 & 9 & 10 & 11\\
\midrule
0 & 0 & 0 & 1 \\
1 & 0 & 4 & 7 & 4 & 1 \\
2 & 0 & 8 & 29 & 56 & 70 & 56 & 28 & 8 & 1 \\
3 & 0 & 12 & 67 & 220 & 495 & 792 & 924 & 792 & 495 & 220 & 66 & \ldots \\% & 12 & 1 \\
4 & 0 & 16 & 121 & 560 & 1820 & 4368 & 8008 & 11440 & 12870 & 11440 & 8008 & \ldots \\% & 4368 & 1820 & 560 & 120 & 16 & 1 \\
5 & 0 & 20 & 191 & 1140 & 4845 & 15504 & 38760 & 77520 & 125970 & 167960 & 184756 & \ldots \\% & 167960 & 125970 & 77520 & 38760 & 15504 & 4845 & 1140 & 190 & 20 & 1	
%	0 & 0 & 0 & 1\\
%	1 & 0 & 1 & 4 & 6 & 4 & 1\\
%	2 & 0 & 8 & 29 & 56 & 70 & 56 & 28 & 8 & 1\\
%	3 & 0 & 48 & 220 & 511 & 804 & 927 & 792 & 495 & 220 & 66 & 12 & 1\\
%	4 & 0 & 256 & 1504 & 4336 & 8273 & 11744 & 13036 & 11488 & 8014 & 4368 & 1820 & 560 & 120 & 16 & 1\\
\end{array}$%}
\caption{$ \cof{x^k}\left((x+1)^{4n}+x^2-1\right) $  (cf.\ \textbf{N}\ref{tangle:A8}).}
\label{Tab:T22}
\end{table}

\begin{table}[H] 
\centering 
$\begin{array}{@{}c|rrrrrrrrrrrr@{}}%
n \backslash k &0 & 1 & 2 & 3 & 4 & 5 & 6 & 7 & 8 & 9 & 10 & 11\\ 
\midrule 
0 & 0 & 1 & 1 \\ 
1 & 0 & 5 & 11 & 10 & 5 & 1 \\ 
2 & 0 & 9 & 37 & 84 & 126 & 126 & 84 & 36 & 9 & 1 \\ 
3 & 0 & 13 & 79 & 286 & 715 & 1287 & 1716 & 1716 & 1287 & 715 & 286  & \ldots \\ % & 78 & 13 & 1 \\ 
4 & 0 & 17 & 137 & 680 & 2380 & 6188 & 12376 & 19448 & 24310 & 24310 & 19448  & \ldots \\ % & 12376 & 6188 & 2380 & 680 & 136 & 17 & 1 \\ 
5 & 0 & 21 & 211 & 1330 & 5985 & 20349 & 54264 & 116280 & 203490 & 293930 & 352716  & \ldots \\ % & 352716 & 293930 & 203490 & 116280 & 54264 & 20349 & 5985 & 1330 & 210 & 21 & 1 \\ 	
\end{array}$ 
\caption{$ \cof{x^k}\left((x+1)^{4n+1}+x^2-1\right) $   (cf.\ \textbf{R}\ref{tangle:A8}).}
\label{Tab:T23}
\end{table}

\begin{table}[H]
\centering
$\begin{array}{@{}c|rrrrrrrrrrrr@{}}%
n \backslash k & 0 & 1 & 2 & 3 & 4 & 5 & 6 & 7 & 8 & 9 & 10 &11\\
\midrule
0 & 0 & 1\\
1 & 0 & 4 & 7 & 4 & 1\\
2 & 0 & 16 & 56 & 81 & 64 & 30 & 8 & 1\\
3 & 0 & 64 & 336 & 780 & 1063 & 948 & 579 & 244 & 69 & 12 & 1\\
4 & 0 & 256 & 1792 & 5728 & 11120 & 14689 & 13984 & 9884 & 5248 & 2086 & 608 & \ldots \\ %124 & 16 & 1\\
5 & 0 & 1024 & 8960 & 36480	& 92000	& 161300 & 208967 &	207300 & 160805 & 98600 & 47910 & \ldots \\%18344	5450	1220	195	20	1
\end{array}$
\caption{$ \cof{x^k}\left(x\left(x^3+4x^2+7x+4\right)^{n} \right) $    (cf.\ \textbf{D}\ref{tangle:A9}, \textbf{D}\ref{tangle:A10}, \textbf{D}\ref{tangle:A17}).}
\label{Tab:T24}	
\end{table}

\begin{table}[H]
\centering
$\begin{array}{@{}c|rrrrrrrrrrr@{}}%
n \backslash k & 0 & 1 & 2 & 3 & 4 & 5 & 6 & 7 & 8 & 9 & 10\\
\midrule
0 & 0 & 0 & 1\\
1 & 0 & 1 & 4 & 6 & 4 & 1\\
2 & 0 & 8 & 29 & 56 & 70 & 56 & 28 & 8 & 1\\
3 & 0 & 48 & 220 & 511 & 804 & 927 & 792 & 495 & 220 & 66 & \ldots \\ %12 &  & 1\\
4 & 0 & 256 & 1504 & 4336 & 8273 & 11744 & 13036 & 11488 & 8014 & 4368  & \ldots \\ %1820 & 560 & 120 & 16 & 1\\
5 & 0 & 1280 & 9344 & 33120 & 76500 & 130151 & 174020 & 189301 & 170120 & 126630& \ldots \\ % & 77640  & 38770 & 15504 & 4845 & 1140 & 190 & 20 & 1
\end{array}$
\caption{$ \cof{x^k}\left(\left(x^3+4x^2+7x+4\right)^{n}+\left(x^2-1\right)\left(x^3+4x^2+6x+4\right)^n\right) $    (cf.\ \textbf{N}\ref{tangle:A9}).}
\label{Tab:T25}	
\end{table}

\begin{table}[H] 
\centering 
$\begin{array}{@{}c|rrrrrrrrrrr@{}}%
n \backslash k&0 & 1 & 2 & 3 & 4 & 5 & 6 & 7 & 8 & 9 & 10 \\ 
\midrule
0 & 0 & 1 & 1\\
1 & 0 & 5 & 11 & 10 & 5 & 1\\
2 & 0 & 24 & 85 & 137 & 134 & 86 & 36 & 9 & 1\\
3 & 0 & 112 & 556 & 1291 & 1867 & 1875 & 1371 & 739 & 289 & 78  & \ldots \\ % & 13 & 1\\
4 & 0 & 512 & 3296 & 10064 & 19393 & 26433 & 27020 & 21372 & 13262 & 6454  & \ldots \\ % & 2428 & 684 & 136 & 17 & 1\\
5 & 0 & 2304 & 18304 & 69600 & 168500 & 291451 & 382987 & 396601 & 330925 & 225230  & \ldots \\ % & 125550 & 57114 & 20954 & 6065 & 1335 & 210 & 21 & 1
\end{array}$ 
\caption{$ \cof{x^k}\left((x+1)\left(x^3+4x^2+7x+4\right)^{n}+\left(x^2-1\right)\left(x^3+4x^2+6x+4\right)^n\right) $    (cf.\ \textbf{R}\ref{tangle:A9}).}
\label{Tab:T26}	
\end{table}

\begin{table}[H]
\centering
$\begin{array}{@{}c|rrrrrrrrrrr@{}}%
n \backslash k &0 & 1 & 2 & 3 & 4 & 5 & 6 & 7 & 8 & 9 & 10\\ 
\midrule
0 & 0 & 0 & 1\\
1 & 0 & 4 & 8 & 4\\
2 & 0 & 32 & 88 & 88 & 39 & 8 & 1\\
3 & 0 & 192 & 736 & 1180 & 1056 & 606 & 244 & 69 & 12 & 1\\
4 & 0 & 1024 & 5120 & 11456 & 15472 & 14416 & 9965 & 5248 & 2086 & 608 & \ldots \\ % & 16 & 1\\ 124  &
5 & 0 & 5120 & 31744 & 91520 & 165440 & 213044 & 208920 & 161048 & 98600 & 47910  & \ldots \\ % & 18344 & 5450 & 1220 & 195 & 20 & 1\\
\end{array}$
\caption{$ \cof{x^k}\left(\left(x^3+4x^2+7x+4\right)^{n}+\left(x^2-1\right)(3x+4)^n\right) $    (cf.\ \textbf{N}\ref{tangle:A10}).}
\label{Tab:T27}	
\end{table}

\begin{table}[H] 
\centering 
$\begin{array}{@{}c|rrrrrrrrrrr@{}}%
n \backslash k &0 & 1 & 2 & 3 & 4 & 5 & 6 & 7 & 8 & 9 & 10\\ 
\midrule 
0 & 0 & 1 & 1 \\ 
1 & 0 & 8 & 15 & 8 & 1 \\ 
2 & 0 & 48 & 144 & 169 & 103 & 38 & 9 & 1 \\ 
3 & 0 & 256 & 1072 & 1960 & 2119 & 1554 & 823 & 313 & 81 & 13 & \ldots \\ %   1 \\ 
4 & 0 & 1280 & 6912 & 17184 & 26592 & 29105 & 23949 & 15132 & 7334 & 2694 & \ldots \\ %  732  && 140 & 17 & 1 \\ 
5 & 0 & 6144 & 40704 & 128000 & 257440 & 374344 & 417887 & 368348 & 259405 & 146510 & \ldots \\ %  66254  && 23794 & 6670 & 1415 & 215 & 21 & 1 \\ 
\end{array}$
\caption{$ \cof{x^k}\left((x+1)\left(x^3+4x^2+7x+4\right)^{n}+\left(x^2-1\right)(3x+4)^n\right) $    (cf.\ \textbf{R}\ref{tangle:A10}).}
\label{Tab:T28}	
\end{table}

\begin{table}[H]
\centering
$\begin{array}{@{}c|rrrrrrrrrrr@{}}%
n \backslash k & 0 & 1 & 2 & 3 & 4 & 5 & 6 & 7 & 8 & 9 & 10\\
\midrule
0 & 0 & 1\\
1 & 0 & 4 & 8 & 4\\
2 & 0 & 16 & 64 & 96 & 64 & 16\\
3 & 0 & 64 & 384 & 960 & 1280 & 960 & 384 & 64\\
4 & 0 & 256 & 2048 & 7168 & 14336 & 17920 & 14336 & 7168 & 2048 & 256\\
5 & 0 & 1024 & 10240 & 46080 & 122880 & 215040 & 258048 & 215040 & 122880 & 46080 &  \ldots \\ %  10240 & 1024\\
\end{array}$
\caption{$ \cof{x^k}\left(x\left(2x+2\right)^{2n}\right) $   (cf.\  \textbf{D}\ref{tangle:A11},  \textbf{D}\ref{tangle:A30}).}
\label{Tab:T29}	
\end{table}

\begin{table}[H]
\centering
$\begin{array}{@{}c|rrrrrrrrrrr@{}}%
n \backslash k &0 & 1 & 2 & 3 & 4 & 5 & 6 & 7 & 8 & 9 & 10\\ 
\midrule
0 & 0 & 0 & 1\\
1 & 0 & 4 & 7 & 4 & 1\\
2 & 0 & 32 & 88 & 88 & 39 & 8 & 1\\
3 & 0 & 192 & 784 & 1312 & 1140 & 532 & 123 & 12 & 1\\
4 & 0 & 1024 & 5632 & 13568 & 18592 & 15680 & 8176 & 2480 & 367 & 16 & \ldots \\ %   1\\
5 & 0 & 5120 & 35584 & 112640 & 213120 & 265344 & 225120 & 129984 & 49260 & 11180 &  \ldots \\ %  1203 & 20 & 1\\
\end{array}$
\caption{$ \cof{x^k}\left(\left(2x+2\right)^{2n}+\left(x^2-1\right)\left(x+2\right)^{2n}\right) $    (cf.\ \textbf{N}\ref{tangle:A11}).}
\label{Tab:T30}	
\end{table}

\begin{table}[H] 
\centering 
$\begin{array}{@{}c|rrrrrrrrrrr@{}}%
n \backslash k &0 & 1 & 2 & 3 & 4 & 5 & 6 & 7 & 8 & 9 & 10\\ 
\midrule 
0 & 0 & 1 & 1\\
1 & 0 & 8 & 15 & 8 & 1\\
2 & 0 & 48 & 152 & 184 & 103 & 24 & 1\\
3 & 0 & 256 & 1168 & 2272 & 2420 & 1492 & 507 & 76 & 1\\
4 & 0 & 1280 & 7680 & 20736 & 32928 & 33600 & 22512 & 9648 & 2415 & 272& \ldots \\ %   & 1\\
5 & 0 & 6144 & 45824 & 158720 & 336000 & 480384 & 483168 & 345024 & 172140 & 57260 & \ldots \\ %   11443 & 1044 & 1\\
\end{array}$
\caption{$ \cof{x^k}\left((x + 1) (2 x + 2)^{2n} + \left(x^2 - 1\right) (x + 2)^{2 n} \right) $    (cf.\ \textbf{R}\ref{tangle:A11}).}
\label{Tab:T31}	
\end{table}

\begin{table}[H]
\centering
$\begin{array}{@{}c|rrrrrrrrrrr@{}}%
n \backslash k &0 & 1 & 2 & 3 & 4 & 5 & 6 & 7 & 8 & 9 & 10\\ 
\midrule
0 & 0 & 1\\
1 & 0 & 5 & 8 & 3\\
2 & 0 & 25 & 80 & 94 & 48 & 9\\
3 & 0 & 125 & 600 & 1185 & 1232 & 711 & 216 & 27\\
4 & 0 & 625 & 4000 & 11100 & 17440 & 16966 & 10464 & 3996 & 864 & 81\\
5 & 0 & 3125 & 25000 & 89375 & 188000 & 257650 & 240368 & 154590 & 67680 & 19305 & \ldots \\ %   3240 & 243\\
\end{array}$
\caption{$ \cof{x^k}\left(x\left(3x^2+8x+5\right)^{n} \right) $    (cf.\ \textbf{D}\ref{tangle:A12}, \textbf{D}\ref{tangle:A14}).}
\label{Tab:T32}	
\end{table}

\begin{table}[H]
\centering
$\begin{array}{@{}c|rrrrrrrrrrr@{}}%
n \backslash k &0 & 1 & 2 & 3 & 4 & 5 & 6 & 7 & 8 & 9 & 10\\ 
\midrule
0 & 0 & 0 & 1\\
1 & 0 & 2 & 6 & 6 & 2\\
2 & 0 & 20 & 63 & 84 & 61 & 24 & 4\\
3 & 0 & 150 & 620 & 1106 & 1125 & 720 & 295 & 72 & 8\\
4 & 0 & 1000 & 5325 & 12520 & 17150 & 15216 & 9188 & 3840 & 1089 & 192 & \ldots \\ %   16\\
5 & 0 & 6250 & 41250 & 122750 & 217500 & 255392 & 209430 & 123216 & 52585 & 16200 & \ldots \\ %   3491 & 480 & 32\\
\end{array}$
\caption{$\cof{x^k}\left(\left(3x^2+8x+5\right)^{n}+\left(x^2-1\right)\left(2x^2+6x+5\right)^n\right) $ (cf.\ \textbf{N}\ref{tangle:A12}).}
\label{Tab:T33}	
\end{table}

\begin{table}[H] 
\centering 
$\begin{array}{@{}c|rrrrrrrrrrr@{}}%
n \backslash k &0 & 1 & 2 & 3 & 4 & 5 & 6 & 7 & 8 & 9 & 10\\ 
\midrule 
0 & 0 & 1 & 1 \\ 
1 & 0 & 7 & 14 & 9 & 2 \\ 
2 & 0 & 45 & 143 & 178 & 109 & 33 & 4 \\ 
3 & 0 & 275 & 1220 & 2291 & 2357 & 1431 & 511 & 99 & 8 \\ 
4 & 0 & 1625 & 9325 & 23620 & 34590 & 32182 & 19652 & 7836 & 1953 & 273 & \ldots \\ %  16 \\ 
5 & 0 & 9375 & 66250 & 212125 & 405500 & 513042 & 449798 & 277806 & 120265 & 35505 & \ldots \\ %  6731 & 723 & 32 \\ 
\end{array}$
\caption{$\cof{x^k}\left((x+1)\left(3x^2+8x+5\right)^{n}+\left(x^2-1\right)\left(2x^2+6x+5\right)^n\right) $ (cf.\ \textbf{R}\ref{tangle:A12}).}
\label{Tab:T34}	
\end{table}

\begin{table}[H]
\centering
$\begin{array}{@{}c|rrrrrrrrrrrr@{}}%
n \backslash k &0 & 1 & 2 & 3 & 4 & 5 & 6 & 7 & 8 & 9 & 10 & 11\\ 
\midrule
0 & 0 & 1\\
1 & 0 & 2 & 6 & 6 & 2\\
2 & 0 & 4 & 24 & 60 & 80 & 60 & 24 & 4\\
3 & 0 & 8 & 72 & 288 & 672 & 1008 & 1008 & 672 & 288 & 72 & 8\\
4 & 0 & 16 & 192 & 1056 & 3520 & 7920 & 12672 & 14784 & 12672 & 7920 & 3520  & \ldots \\ % & 1056 & 192 & 16\\
5 & 0 & 32 & 480 & 3360 & 14560 & 43680 & 96096 & 160160 & 205920 & 205920 & 160160  & \ldots \\ % & 96096 & 43680 & 14560 & 3360 & 480 & 32\\
\end{array}$
\caption{$\cof{x^k}\left(x\left(2x^3+6x^2+6x+2\right)^n\right) $ (cf.\ \textbf{D}\ref{tangle:A13}, \textbf{D}\ref{tangle:A23}, \textbf{D}\ref{tangle:A27}, \textbf{D}\ref{tangle:A29}, \textbf{D}\ref{tangle:A33}).}
\label{Tab:T35}	
\end{table}

\begin{table}[H]
\centering
$\begin{array}{@{}c|rrrrrrrrrrrr@{}}%
n \backslash k &0 & 1 & 2 & 3 & 4 & 5 & 6 & 7 & 8 & 9 & 10 & 11\\ 
\midrule
0 & 0 & 0 & 1\\
1 & 0 & 5 & 8 & 3\\
2 & 0 & 20 & 63 & 84 & 61 & 24 & 4\\
3 & 0 & 60 & 290 & 683 & 1014 & 1009 & 672 & 288 & 72 & 8\\
4 & 0 & 160 & 1048 & 3544 & 7943 & 12680 & 14785 & 12672 & 7920 & 3520 & 1056 & \ldots \\ % & 192 & 16\\
5 & 0 & 400 & 3312 & 14600 & 43750 & 96135 & 160170 & 205921 & 205920 & 160160 & 96096 & \ldots\\ % & 43680 & 14560 & 3360 & 480 & 32\\
\end{array}$
\caption{$\cof{x^k}\left(\left(2x^3+6x^2+6x+2\right)^n+\left(x^2-1\right)(x+2)^n\right) $ (cf.\ \textbf{N}\ref{tangle:A13}).}
\label{Tab:T36}	
\end{table}

\begin{table}[H] 
\centering 
$\begin{array}{@{}c|rrrrrrrrrrr@{}}%
n \backslash k &0 & 1 & 2 & 3 & 4 & 5 & 6 & 7 & 8 & 9 & 10 \\ 
\midrule 
0 & 0 & 1 & 1 \\ 
1 & 0 & 7 & 14 & 9 & 2 \\ 
2 & 0 & 24 & 87 & 144 & 141 & 84 & 28 & 4 \\ 
3 & 0 & 68 & 362 & 971 & 1686 & 2017 & 1680 & 960 & 360 & 80 & \ldots \\ %   8 \\ 
4 & 0 & 176 & 1240 & 4600 & 11463 & 20600 & 27457 & 27456 & 20592 & 11440 & \ldots \\ %  4576  && 1248 & 208 & 16 \\ 
5 & 0 & 432 & 3792 & 17960 & 58310 & 139815 & 256266 & 366081 & 411840 & 366080 & \ldots \\ %  256256  && 139776 & 58240 & 17920 & 3840 & 512 & 32 \\ 
\end{array}$
\caption{$\cof{x^k}\left((x+1)\left(2x^3+6x^2+6x+2\right)^n+\left(x^2-1\right)(x+2)^n\right) $  (cf.\ \textbf{R}\ref{tangle:A13}).}
\label{Tab:T37}		
\end{table}

\begin{table}[H]
\centering
$\begin{array}{@{}c|rrrrrrrrrrrr@{}}%
n \backslash k & 0 & 1 & 2 & 3 & 4 & 5 & 6 & 7 & 8 & 9 & 10 & 11\\
\midrule
0 & 0 & 0 & 1\\
1 & 0 & 3 & 7 & 5 & 1\\
2 & 0 & 30 & 84 & 88 & 43 & 10 & 1\\
3 & 0 & 225 & 860 & 1332 & 1071 & 476 & 116 & 15 & 1\\
4 & 0 & 1500 & 7475 & 15940 & 18941 & 13664 & 6101 & 1644 & 250 & 20 & 1\\
5 & 0 & 9375 & 58125 & 159875 & 256400 & 264743 & 183090 & 85305 & 26155 & 4965 & 517 & \ldots \\ %  25 & 1\\
\end{array}$
\caption{$\cof{x^k}\left(\left(3 x^2 + 8 x + 5\right)^n + \left(x^2 - 1\right) \left(x^2 + 5 x + 5\right)^n\right) $ (cf.\ \textbf{N}\ref{tangle:A14}).}
\label{Tab:T38}	
\end{table}

\begin{table}[H] 
\centering 
$\begin{array}{@{}c|rrrrrrrrrrr@{}}%
n \backslash k &0 & 1 & 2 & 3 & 4 & 5 & 6 & 7 & 8 & 9 & 10\\ 
\midrule 
0 & 0 & 1 & 1 \\ 
1 & 0 & 8 & 15 & 8 & 1 \\ 
2 & 0 & 55 & 164 & 182 & 91 & 19 & 1 \\ 
3 & 0 & 350 & 1460 & 2517 & 2303 & 1187 & 332 & 42 & 1 \\ 
4 & 0 & 2125 & 11475 & 27040 & 36381 & 30630 & 16565 & 5640 & 1114 & 101 & \ldots \\ %  1 \\ 
5 & 0 & 12500 & 83125 & 249250 & 444400 & 522393 & 423458 & 239895 & 93835 & 24270 & \ldots \\ %  3757 & 268 & 1 \\ 
\end{array}$
\caption{$\cof{x^k}\left((x+1)\left(3 x^2 + 8 x + 5\right)^n + \left(x^2 - 1\right) \left(x^2 + 5 x + 5\right)^n\right) $ (cf.\ \textbf{R}\ref{tangle:A14}).}
\label{Tab:T39}		
\end{table}

\begin{table}[H]
\centering
$\begin{array}{@{}c|rrrrrrrrrrrr@{}}%
n \backslash k &0 & 1 & 2 & 3 & 4 & 5 & 6 & 7 & 8 & 9 & 10 & 11\\ 
\midrule
0 & 0 & 1\\
1 & 0 & 3 & 7 & 5 & 1\\
2 & 0 & 9 & 42 & 79 & 76 & 39 & 10 & 1\\
3 & 0 & 27 & 189 & 576 & 1000 & 1086 & 762 & 344 & 96 & 15 & \ldots\\%  1\\
4 & 0 & 81 & 756 & 3186 & 8004 & 13327 & 15464 & 12796 & 7592 & 3199 & \ldots\\%  932 && 178 & 20 & 1\\
5 & 0 & 243 & 2835 & 15255 & 50175 & 112695 & 182887 & 221275 & 202995 & 142185 & \ldots\\%  75945 && 30645 & 9165 & 1965 & 285 & 25 & 1\\
\end{array}$
\caption{$ \cof{x^k} \left(x\left(x^3 + 5 x^2 + 7 x + 3\right)^n\right)$ (cf.\ \textbf{D}\ref{tangle:A15}, \textbf{D}\ref{tangle:A16}, \textbf{D}\ref{tangle:A22}, \textbf{D}\ref{tangle:A24}, \textbf{D}\ref{tangle:A32}).}
\label{Tab:T40}
\end{table}

\begin{table}[H]
\centering
$\begin{array}{@{}c|rrrrrrrrrrr@{}}%
n \backslash k &0 & 1 & 2 & 3 & 4 & 5 & 6 & 7 & 8 & 9 & 10\\ 
\midrule
0 & 0 & 0 & 1\\
1 & 0 & 5 & 8 & 3\\
2 & 0 & 30 & 84 & 88 & 43 & 10 & 1\\
3 & 0 & 135 & 567 & 1046 & 1122 & 770 & 344 & 96 & 15 & 1\\
4 & 0 & 540 & 3051 & 8124 & 13527 & 15560 & 12812 & 7592 & 3199 & 932 &\ldots\\ %  & 178 & 20 & 1\\
5 & 0 & 2025 & 14418 & 50265 & 113535 & 183575 & 221515 & 203027 & 142185 & 75945 &\ldots\\ %  & 30645 & 9165 & 1965 & 285 & 25 & 1\\
\end{array}$
\caption{$ \cof{x^k} \left( \left(x^3+5x^2+7x+3\right)^{n}+\left(x^2-1\right)\left(2x+3\right)^n \right)$ (cf.\ \textbf{N}\ref{tangle:A15}).}
\label{Tab:T41}
\end{table}

\begin{table}[H]
\centering
$\begin{array}{@{}c|rrrrrrrrrrr@{}}%
n \backslash k &0 & 1 & 2 & 3 & 4 & 5 & 6 & 7 & 8 & 9 & 10\\ 
\midrule
0 & 0 & 1 & 1\\
1 & 0 & 8 & 15 & 8 & 1\\
2 & 0 & 39 & 126 & 167 & 119 & 49 & 11 & 1\\
3 & 0 & 162 & 756 & 1622 & 2122 & 1856 & 1106 & 440 & 111 & 16 & \ldots\\ % 1\\
4 & 0 & 621 & 3807 & 11310 & 21531 & 28887 & 28276 & 20388 & 10791 & 4131 &\ldots\\ % & 1110 & 198 & 21 & 1\\
5 & 0 & 2268 & 17253 & 65520 & 163710 & 296270 & 404402 & 424302 & 345180 & 218130 & \ldots\\ % & 106590 & 39810 & 11130 & 2250 & 310 & 26 & 1\\
\end{array}$
\caption{$ \cof{x^k} \left( (x+1)\left(x^3+5x^2+7x+3\right)^{n}+\left(x^2-1\right)\left(2x+3\right)^n \right)$ (cf.\ \textbf{R}\ref{tangle:A15}).}
\label{Tab:T42}
\end{table}

\begin{table}[H]
\centering
$\begin{array}{@{}c|rrrrrrrrrrr@{}}%
n \backslash k &0 & 1 & 2 & 3 & 4 & 5 & 6 & 7 & 8 & 9 & 10\\ 
\midrule
0 & 0 & 0 & 1\\
1 & 0 & 4 & 7 & 4 & 1\\
2 & 0 & 24 & 73 & 88 & 53 & 16 & 2\\
3 & 0 & 108 & 495 & 1000 & 1158 & 834 & 379 & 105 & 16 & 1\\
4 & 0 & 432 & 2673 & 7680 & 13462 & 15896 & 13189 & 7796 & 3264 & 944 & \ldots \\ %  179  && 20 & 1\\
5 & 0 & 1620 & 12663 & 47340 & 111615 & 184264 & 223885 & 205218 & 143385 & 76380 & \ldots \\ %   30749  && 9180  & \ldots \\ % & 1966 & 285 & 25 & 1
\end{array}$
\caption{$ \cof{x^k} \left(\left(x^3 + 5 x^2 + 7 x + 3\right)^n + \left(x^2 - 1\right) \left(x^2 + 3 x + 3\right)^n\right)$ (cf.\ \textbf{N}\ref{tangle:A16}).}
\label{Tab:T43}	
\end{table}

\begin{table}[H] 
\centering 
$\begin{array}{@{}c|rrrrrrrrrrr@{}}%
n \backslash k &0 & 1 & 2 & 3 & 4 & 5 & 6 & 7 & 8 & 9 & 10 \\ 
\midrule 
0 & 0 & 1 & 1 \\ 
1 & 0 & 7 & 14 & 9 & 2 \\ 
2 & 0 & 33 & 115 & 167 & 129 & 55 & 12 & 1 \\ 
3 & 0 & 135 & 684 & 1576 & 2158 & 1920 & 1141 & 449 & 112 & 16 & \ldots \\ %  1 \\ 
4 & 0 & 513 & 3429 & 10866 & 21466 & 29223 & 28653 & 20592 & 10856 & 4143 & \ldots \\ %  1111  && 198 & 21 & 1 \\ 
5 & 0 & 1863 & 15498 & 62595 & 161790 & 296959 & 406772 & 426493 & 346380 & 218565 & \ldots \\ %  106694  && 39825 & 11131 & 2250 & 310 & 26 & 1 \\ 
\end{array}$
\caption{$ \cof{x^k} \left((x + 1) \left(x^3 + 5 x^2 + 7 x + 3\right)^n + \left(x^2 - 1\right) \left(x^2 + 3 x + 3\right)^n\right)$  (cf.\ \textbf{R}\ref{tangle:A16}).}
\label{Tab:T44}
\end{table}

\begin{table}[H]
\centering
$\begin{array}{@{}c|rrrrrrrrrrr@{}}%
n \backslash k & 0 & 1 & 2 & 3 & 4 & 5 & 6 & 7 & 8 & 9 & 10\\
\midrule
0 & 0 & 0 & 1\\
1 & 0 & 3 & 7 & 5 & 1\\
2 & 0 & 24 & 73 & 88 & 53 & 16 & 2\\
3 & 0 & 144 & 604 & 1095 & 1128 & 727 & 303 & 81 & 13 & 1\\
4 & 0 & 768 & 4192 & 10352 & 15361 & 15328 & 10892 & 5680 & 2197 & 624 & \ldots \\ % 125  & & 16 & 1\\
5 & 0 & 3840 & 25984 & 81760 & 159380 & 216263 & 217380 & 167909 & 101780 & 48850 & \ldots \\ %   18523  && 5470 & 1221 & 195 & 20 & 1\\
\end{array}$
\caption{$ \cof{x^k} \left(\left(x^3 + 4 x^2 + 7 x + 4\right)^n + \left(x^2 - 1\right) \left(x^2 + 4 x + 4\right)^n\right)$ (cf.\ \textbf{N}\ref{tangle:A17}).}
\label{Tab:T45}
\end{table}

\begin{table}[H] 
\centering 
$\begin{array}{@{}c|rrrrrrrrrrr@{}}%
n \backslash k &0 & 1 & 2 & 3 & 4 & 5 & 6 & 7 & 8 & 9 & 10\\ 
\midrule 
0 & 0 & 1 & 1 \\ 
1 & 0 & 7 & 14 & 9 & 2 \\ 
2 & 0 & 40 & 129 & 169 & 117 & 46 & 10 & 1 \\ 
3 & 0 & 208 & 940 & 1875 & 2191 & 1675 & 882 & 325 & 82 & 13 & \ldots \\ % 1 \\ 
4 & 0 & 1024 & 5984 & 16080 & 26481 & 30017 & 24876 & 15564 & 7445 & 2710 & \ldots \\ % 733  && 140 & 17 & 1 \\ 
5 & 0 & 4864 & 34944 & 118240 & 251380 & 377563 & 426347 & 375209 & 262585 & 147450 & \ldots \\ % 66433  && 23814 & 6671 & 1415 & 215 & 21 & 1 \\ 
\end{array}$
\caption{$ \cof{x^k} \left((x + 1) \left(x^3 + 4 x^2 + 7 x + 4\right)^n + \left(x^2 - 1\right) \left(x^2 + 4 x + 4\right)^n\right)$ (cf.\ \textbf{R}\ref{tangle:A17}).}
\label{Tab:T46}
\end{table}

\begin{table}[H]
\centering
$\begin{array}{@{}c|rrrrrrrrrrr@{}}%
n \backslash k & 0 & 1 & 2 & 3 & 4 & 5 & 6 & 7 & 8 & 9 & 10\\
\midrule
0 & 0 & 0 & 1\\
1 & 0 & 1 & 2 & 1\\
2 & 0 & 2 & 6 & 6 & 2\\
3 & 0 & 3 & 13 & 22 & 18 & 7 & 1\\
4 & 0 & 4 & 23 & 56 & 75 & 60 & 29 & 8 & 1\\
5 & 0 & 5 & 36 & 115 & 215 & 261 & 215 & 121 & 45 & 10 & 1\\
\end{array}$
\caption{$ \cof{x^k} \left((x+1)^{2n}+\left(x^2-1\right)(x+1)^n\right) $ (cf.\ \textbf{N}\ref{tangle:A18}) \cite[\seqnum{A300192}]{Sloane}.}
\label{Tab:T47}
\end{table}

\begin{table}[H] 
\centering 
$\begin{array}{@{}c|rrrrrrrrrrrr@{}}%
n \backslash k &0 & 1 & 2 & 3 & 4 & 5 & 6 & 7 & 8 & 9 & 10 & 11 \\ 
\midrule 
0 & 0 & 1 & 1 \\ 
1 & 0 & 2 & 4 & 2 \\ 
2 & 0 & 3 & 10 & 12 & 6 & 1 \\ 
3 & 0 & 4 & 19 & 37 & 38 & 22 & 7 & 1 \\ 
4 & 0 & 5 & 31 & 84 & 131 & 130 & 85 & 36 & 9 & 1 \\ 
5 & 0 & 6 & 46 & 160 & 335 & 471 & 467 & 331 & 165 & 55 & 11 & 1 \\ 
\end{array}$ 
\caption{$ \cof{x^k} \left(\left((x+1)^{2n+1}+\left(x^2-1\right)(x+1)^n\right)\right) $ (cf.\ \textbf{R}\ref{tangle:A18}).}
\label{Tab:T48}	
\end{table}

\begin{table}[H]
\centering
$\begin{array}{@{}c|rrrrrrrrrrrrrrr@{}}%
n \backslash k & 0 & 1 & 2 & 3 & 4 & 5 & 6 & 7 & 8 & 9 & 10 & 11 &12 & 13 & 14\\
\midrule
0 & 0 & 0 & 1\\
1 & 0 & 2 & 4 & 2\\
2 & 0 & 4 & 15 & 22 & 16 & 6 & 1\\
3 & 0 & 6 & 34 & 86 & 129 & 127 & 84 & 36 & 9 & 1\\
4 & 0 & 8 & 61 & 220 & 500 & 796 & 925 & 792 & 495 & 220 & 66 & 12 & 1\\
5 & 0 & 10 & 96 & 450 & 1370 & 3012 & 5010 & 6436 & 6435 & 5005 & 3003 & 1365 & 455 & 105 & \ldots \\ % & 15 & 1
\end{array}$
\caption{$ \cof{x^k} \left((x+1)^{3n}+\left(x^2-1\right)(x+1)^n\right) $  (cf.\ \textbf{N}\ref{tangle:A19}).}
\label{Tab:T49}		
\end{table}	

\begin{table}[H] 
\centering 
$\begin{array}{@{}c|rrrrrrrrrrrrrr@{}}%
n \backslash k &0 & 1 & 2 & 3 & 4 & 5 & 6 & 7 & 8 & 9 & 10 & 11 & 12 & 13\\ 
\midrule 
0 & 0 & 1 & 1 \\ 
1 & 0 & 3 & 7 & 5 & 1 \\ 
2 & 0 & 5 & 21 & 37 & 36 & 21 & 7 & 1 \\ 
3 & 0 & 7 & 43 & 122 & 213 & 253 & 210 & 120 & 45 & 10 & 1 \\ 
4 & 0 & 9 & 73 & 286 & 720 & 1291 & 1717 & 1716 & 1287 & 715 & 286 & 78 & 13 & \ldots \\ % & 1 \\ 
5 & 0 & 11 & 111 & 555 & 1825 & 4377 & 8013 & 11441 & 12870 & 11440 & 8008 & 4368 & 1820 & \ldots \\ % & 560 & 120 & 16 & 1 \\ 
\end{array}$ 
\caption{$ \cof{x^k} \left((x+1)^{3n+1}+\left(x^2-1\right)(x+1)^n\right) $  (cf.\ \textbf{R}\ref{tangle:A19}).}
\label{Tab:T50}		
\end{table}

\begin{table}[H]
\centering
$\begin{array}{@{}c|rrrrrrrrrrrrr@{}}%
n \backslash k & 0 & 1 & 2 & 3 & 4 & 5 & 6 & 7 & 8 & 9 & 10 & 11 &12\\
\midrule
0 & 0 & 0 & 1\\
1 & 0 & 1 & 3 & 3 & 1\\
2 & 0 & 4 & 15 & 22 & 16 & 6 & 1\\
3 & 0 & 12 & 62 & 133 & 153 & 102 & 40 & 9 & 1\\
4 & 0 & 32 & 216 & 632 & 1047 & 1076 & 707 & 296 & 77 & 12 & 1\\
5 & 0 & 80 & 672 & 2520 & 5550 & 7941 & 7705 & 5133 & 2325 & 695 & 131 & 15 & 1\\
\end{array}$
\caption{$ \cof{x^k} \left(\left(2x^2+4x+2\right)^n+\left(x^2-1\right)\left(x^2+3x+2\right)^n\right) $  (cf.\ \textbf{N}\ref{tangle:A20}).}
\label{Tab:T51}		
\end{table}	

\begin{table}[H] 
\centering 
$\begin{array}{@{}c|rrrrrrrrrrrrr@{}}%
n \backslash k &0 & 1 & 2 & 3 & 4 & 5 & 6 & 7 & 8 & 9 & 10 & 11 & 12 \\ 
\midrule 
0 & 0 & 1 & 1 \\ 
1 & 0 & 3 & 7 & 5 & 1 \\ 
2 & 0 & 8 & 31 & 46 & 32 & 10 & 1 \\ 
3 & 0 & 20 & 110 & 253 & 313 & 222 & 88 & 17 & 1 \\ 
4 & 0 & 48 & 344 & 1080 & 1943 & 2196 & 1603 & 744 & 205 & 28 & 1 \\ 
5 & 0 & 112 & 992 & 3960 & 9390 & 14661 & 15769 & 11853 & 6165 & 2135 & 451 & 47 & 1 \\ 
\end{array}$ 
\caption{$ \cof{x^k} \left((x+1)\left(2x^2+4x+2\right)^n+\left(x^2-1\right)\left(x^2+3x+2\right)^n\right) $  (cf.\ \textbf{R}\ref{tangle:A20}).}
\label{Tab:T52}		
\end{table}

\begin{table}[H] 
\centering 
$\begin{array}{@{}c|rrrrrrrrrrrr@{}}%
n \backslash k &0 & 1 & 2 & 3 & 4 & 5 & 6 & 7 & 8 & 9 & 10 & 11 \\ 
\midrule 
0 & 0 & 0 & 1 \\ 
1 & 0 & 3 & 7 & 5 & 1 \\ 
2 & 0 & 6 & 28 & 58 & 71 & 56 & 28 & 8 & 1 \\ 
3 & 0 & 9 & 64 & 222 & 498 & 793 & 924 & 792 & 495 & 220 & 66 & \ldots \\% & 12 & 1 \\ 
4 & 0 & 12 & 115 & 560 & 1825 & 4372 & 8009 & 11440 & 12870 & 11440 & 8008 & \ldots \\% & 4368 & 1820 & 560 & 120 & 16 & 1 \\ 
5 & 0 & 15 & 181 & 1135 & 4850 & 15513 & 38765 & 77521 & 125970 & 167960 & 184756 & \ldots \\% 167960 & 125970 & 77520 & 38760 & 15504 & 4845 & 1140 & 190 & 20 & 1 \\ 
\end{array}$ 
\caption{$ \cof{x^k} \left((x+1)^{4n}+\left(x^2-1\right)(x+1)^n\right) $  (cf.\ \textbf{N}\ref{tangle:A21}).}
\label{Tab:T53}		
\end{table}

\begin{table}[H] 
\centering 
$\begin{array}{@{}c|rrrrrrrrrrrr@{}}%
n \backslash k &0 & 1 & 2 & 3 & 4 & 5 & 6 & 7 & 8 & 9 & 10 & 11 \\ 
\midrule 
0 & 0 & 1 & 1 \\ 
1 & 0 & 4 & 11 & 11 & 5 & 1 \\ 
2 & 0 & 7 & 36 & 86 & 127 & 126 & 84 & 36 & 9 & 1 \\ 
3 & 0 & 10 & 76 & 288 & 718 & 1288 & 1716 & 1716 & 1287 & 715 & 286 & \ldots \\% & 78 & 13 & 1 \\ 
4 & 0 & 13 & 131 & 680 & 2385 & 6192 & 12377 & 19448 & 24310 & 24310 & 19448 & \ldots \\% & 12376 & 6188 & 2380 & 680 & 136 & 17 & 1 \\ 
5 & 0 & 16 & 201 & 1325 & 5990 & 20358 & 54269 & 116281 & 203490 & 293930 & 352716 & \ldots \\% & 352716 & 293930 & 203490 & 116280 & 54264 & 20349 & 5985 & 1330 & 210 & 21 & 1 \\ 
\end{array}$ 
\caption{$ \cof{x^k} \left((x+1)^{4n+1}+\left(x^2-1\right)(x+1)^n\right) $ (cf.\ \textbf{R}\ref{tangle:A21}).}
\label{Tab:T54}		
\end{table}

\begin{table}[H] 
\centering 
$\begin{array}{@{}c|rrrrrrrrrrr@{}}%
n \backslash k &0 & 1 & 2 & 3 & 4 & 5 & 6 & 7 & 8 & 9 & 10 \\ 
\midrule 
0 & 0 & 0 & 1 \\ 
1 & 0 & 1 & 4 & 6 & 4 & 1 \\ 
2 & 0 & 6 & 28 & 58 & 71 & 56 & 28 & 8 & 1 \\ 
3 & 0 & 27 & 171 & 487 & 834 & 969 & 811 & 498 & 220 & 66 & \ldots \\% & 12 & 1 \\ 
4 & 0 & 108 & 891 & 3360 & 7711 & 12116 & 13918 & 12176 & 8301 & 4432 & \ldots \\% & 1826 & 560 & 120 & 16 & 1 \\ 
5 & 0 & 405 & 4158 & 19800 & 58155 & 118276 & 177445 & 204471 & 186135 & 136515 & \ldots \\% & 81581 & 39775 & 15654 & 4855 & 1140 & 190 & 20 & 1 \\ 
\end{array}$ 
\caption{$ \cof{x^k} \left(\left(x^3+5x^2+7x+3\right)^n+\left(x^2-1\right)\left(x^3+4x^2+6x+3\right)^n\right) $  (cf.\ \textbf{N}\ref{tangle:A22}).}
\label{Tab:T55}		
\end{table}

\begin{table}[H] 
\centering 
$\begin{array}{@{}c|rrrrrrrrrrr@{}}%
n \backslash k &0 & 1 & 2 & 3 & 4 & 5 & 6 & 7 & 8 & 9 & 10\\ 
\midrule 
0 & 0 & 1 & 1 \\ 
1 & 0 & 4 & 11 & 11 & 5 & 1 \\ 
2 & 0 & 15 & 70 & 137 & 147 & 95 & 38 & 9 & 1 \\ 
3 & 0 & 54 & 360 & 1063 & 1834 & 2055 & 1573 & 842 & 316 & 81 & \ldots \\% & 13 & 1 \\ 
4 & 0 & 189 & 1647 & 6546 & 15715 & 25443 & 29382 & 24972 & 15893 & 7631 & \ldots \\% & 2758 & 738 & 140 & 17 & 1 \\ 
5 & 0 & 648 & 6993 & 35055 & 108330 & 230971 & 360332 & 425746 & 389130 & 278700 & \ldots \\% & 157526 & 70420 & 24819 & 6820 & 1425 & 215 & 21 & 1 \\ 
\end{array}$ 
\caption{$ \cof{x^k} \left((x+1)\left(x^3+5x^2+7x+3\right)^n+\left(x^2-1\right)\left(x^3+4x^2+6x+3\right)^n\right) $ (cf.\ \textbf{R}\ref{tangle:A22}).}
\label{Tab:T56}		
\end{table}

\begin{table}[H] 
\centering 
$\begin{array}{@{}c|rrrrrrrrrrr@{}}%
n \backslash k &0 & 1 & 2 & 3 & 4 & 5 & 6 & 7 & 8 & 9 & 10 \\ 
%	 	 n \backslash k &0 & 1 & 2 & 3 & 4 & 5 & 6 & 7 & 8 & 9 & 10 & 11 & 12 & 13 & 14 & 15 \\ 
\midrule 
0 & 0 & 0 & 1 \\ 
1 & 0 & 3 & 7 & 5 & 1 \\ 
2 & 0 & 12 & 51 & 86 & 72 & 30 & 5 \\ 
3 & 0 & 36 & 230 & 645 & 1041 & 1062 & 704 & 297 & 73 & 8 \\ 
4 & 0 & 96 & 824 & 3256 & 7847 & 12852 & 15043 & 12840 & 7981 & 3532 & \ldots \\% & 1057 & 192 & 16 \\ 
5 & 0 & 240 & 2592 & 13240 & 42510 & 95973 & 161145 & 207213 & 206805 & 160535 & \ldots \\% & 96195 & 43695 & 14561 & 3360 & 480 & 32 \\ 
\end{array}$ 
\caption{$ \cof{x^k} \left(\left(2x^3+6x^2+6x+2\right)^n+\left(x^2-1\right)\left(x^2+3x+2\right)^n\right) $ (cf.\ \textbf{N}\ref{tangle:A23}).}
\label{Tab:T57}		
\end{table}

\begin{table}[H] 
\centering 
$\begin{array}{@{}c|rrrrrrrrrrr@{}}%
n \backslash k &0 & 1 & 2 & 3 & 4 & 5 & 6 & 7 & 8 & 9 & 10\\ 
\midrule 
0 & 0 & 1 & 1 \\ 
1 & 0 & 5 & 13 & 11 & 3 \\ 
2 & 0 & 16 & 75 & 146 & 152 & 90 & 29 & 4 \\ 
3 & 0 & 44 & 302 & 933 & 1713 & 2070 & 1712 & 969 & 361 & 80 & \ldots \\% & 8 \\ 
4 & 0 & 112 & 1016 & 4312 & 11367 & 20772 & 27715 & 27624 & 20653 & 11452 & \ldots \\% & 4577 & 1248 & 208 & 16 \\ 
5 & 0 & 272 & 3072 & 16600 & 57070 & 139653 & 257241 & 367373 & 412725 & 366455 & \ldots \\% & 256355 & 139791 & 58241 & 17920 & 3840 & 512 & 32 \\ 
\end{array}$ 
\caption{$ \cof{x^k} \left((x+1)\left(2x^3+6x^2+6x+2\right)^n+\left(x^2-1\right)\left(x^2+3x+2\right)^n\right) $ (cf.\ \textbf{R}\ref{tangle:A23}).}
\label{Tab:T58}		
\end{table}

\begin{table}[H]
\centering
$\begin{array}{@{}c|rrrrrrrrrrrr@{}}%
n \backslash k & 0 & 1 & 2 & 3 & 4 & 5 & 6 & 7 & 8 & 9 & 10 & 11\\
\midrule
0 & 0 & 0 & 1\\
1 & 0 & 2 & 6 & 6 & 2\\
2 & 0 & 12 & 51 & 86 & 72 & 30 & 5\\
3 & 0 & 54 & 324 & 830 & 1179 & 1007 & 522 & 156 & 23 & 1\\
4 & 0 & 216 & 1701 & 5964 & 12252 & 16324 & 14741 & 9152 & 3879 & 1092 & 194 & \ldots \\ % & 20 & 1\\
5 & 0 & 810 & 7938 & 35550 & 96300 & 176012 & 229260 & 219120 & 155915 & 82945 & 32853  & \ldots \\ %& 9565 & 1997 & 285 & 25 & 1\\
\end{array}$
\caption{$ \cof{x^k} \left(\left(x^3+5x^2+7x+3\right)^n+\left(x^2-1\right)\left(2x^2+5x+3\right)^n\right) $  (cf.\ \textbf{N}\ref{tangle:A24}).}
\label{Tab:T59}		
\end{table}	

\begin{table}[H] 
\centering 
$\begin{array}{@{}c|rrrrrrrrrrr@{}}%
n \backslash k &0 & 1 & 2 & 3 & 4 & 5 & 6 & 7 & 8 & 9 & 10\\ 
\midrule 
0 & 0 & 1 & 1 \\ 
1 & 0 & 5 & 13 & 11 & 3 \\ 
2 & 0 & 21 & 93 & 165 & 148 & 69 & 15 & 1 \\ 
3 & 0 & 81 & 513 & 1406 & 2179 & 2093 & 1284 & 500 & 119 & 16 & \ldots \\% & 1 \\ 
4 & 0 & 297 & 2457 & 9150 & 20256 & 29651 & 30205 & 21948 & 11471 & 4291 & \ldots \\% & 1126 & 198 & 21 & 1 \\ 
5 & 0 & 1053 & 10773 & 50805 & 146475 & 288707 & 412147 & 440395 & 358910 & 225130 & \ldots \\% & 108798 & 40210 & 11162 & 2250 & 310 & 26 & 1 \\ 
\end{array}$ 
\caption{$ \cof{x^k} \left((x+1)\left(x^3+5x^2+7x+3\right)^n+\left(x^2-1\right)\left(2x^2+5x+3\right)^n\right) $   (cf.\ \textbf{R}\ref{tangle:A24}).}
\label{Tab:T60}		
\end{table}

\begin{table}[H]
\centering
$\begin{array}{@{}c|rrrrrrrrrrrrrrr@{}}%
n \backslash k & 0 & 1 & 2 & 3 & 4 & 5 & 6 & 7 & 8 & 9 & 10 & 11 & 12 & 13 & 14\\
\midrule
0 & 0 & 0 & 1\\
1 & 0 & 1 & 3 & 3 & 1\\
2 & 0 & 2 & 10 & 20 & 20 & 10 & 2\\
3 & 0 & 3 & 22 & 70 & 126 & 140 & 98 & 42 & 10 & 1\\
4 & 0 & 4 & 39 & 172 & 453 & 792 & 966 & 840 & 522 & 228 & 67 & 12 & 1\\
5 & 0 & 5 & 61 & 345 & 1200 & 2871 & 5005 & 6567 & 6600 & 5115 & 3047 & 1375 & 456 & 105 & \ldots \\ %  15 &1\\
\end{array}$
\caption{$ \cof{x^k} \left((x+1)^{3n}+\left(x^2-1\right)(x+1)^{2n}\right) $ (cf.\ \textbf{N}\ref{tangle:A25}).}
\label{Tab:T61}	
\end{table}	

\begin{table}[H] 
\centering 
$\begin{array}{@{}c|rrrrrrrrrrrrrr@{}}%
n \backslash k &0 & 1 & 2 & 3 & 4 & 5 & 6 & 7 & 8 & 9 & 10 & 11 & 12 & 13\\ 
\midrule 
0 & 0 & 1 & 1 \\ 
1 & 0 & 2 & 6 & 6 & 2 \\ 
2 & 0 & 3 & 16 & 35 & 40 & 25 & 8 & 1 \\ 
3 & 0 & 4 & 31 & 106 & 210 & 266 & 224 & 126 & 46 & 10 & 1 \\ 
4 & 0 & 5 & 51 & 238 & 673 & 1287 & 1758 & 1764 & 1314 & 723 & 287 & 78 & 13 & \ldots \\ % 1 \\ 
5 & 0 & 6 & 76 & 450 & 1655 & 4236 & 8008 & 11572 & 13035 & 11550 & 8052 & 4378 & 1821 & \ldots \\ %  560 && 120 & 16 & 1 \\ 
\end{array}$
\caption{$ \cof{x^k} \left((x+1)^{3n+1}+\left(x^2-1\right)(x+1)^{2n}\right) $ (cf.\ \textbf{R}\ref{tangle:A25}).}
\label{Tab:T62}	
\end{table}

\begin{table}[H] 
\centering 
$\begin{array}{@{}c|rrrrrrrrrrrr@{}}%
n \backslash k &0 & 1 & 2 & 3 & 4 & 5 & 6 & 7 & 8 & 9 & 10 & 11\\ 
\midrule 
0 & 0 & 0 & 1 \\ 
1 & 0 & 2 & 6 & 6 & 2 \\ 
2 & 0 & 4 & 23 & 56 & 75 & 60 & 29 & 8 & 1 \\ 
3 & 0 & 6 & 52 & 206 & 495 & 806 & 938 & 798 & 496 & 220 & 66 & \ldots \\% & 12 & 1 \\ 
4 & 0 & 8 & 93 & 512 & 1778 & 4368 & 8050 & 11488 & 12897 & 11448 & 8009 & \ldots \\% & 4368 & 1820 & 560 & 120 & 16 & 1 \\ 
5 & 0 & 10 & 146 & 1030 & 4680 & 15372 & 38760 & 77652 & 126135 & 168070 & 184800 & \ldots \\% & 167970 & 125971 & 77520 & 38760 & 15504 & 4845 & 1140 & 190 & 20 & 1 \\ 
\end{array}$
\caption{$ \cof{x^k} \left((x+1)^{4n}+\left(x^2-1\right)(x+1)^{2n}\right) $  (cf.\ \textbf{N}\ref{tangle:A26}).}
\label{Tab:T63}		 
\end{table}

\begin{table}[H] 
\centering 
$\begin{array}{@{}c|rrrrrrrrrrrr@{}}%
n \backslash k &0 & 1 & 2 & 3 & 4 & 5 & 6 & 7 & 8 & 9 & 10 & 11 \\ 
\midrule 
0 & 0 & 1 & 1 \\ 
1 & 0 & 3 & 10 & 12 & 6 & 1 \\ 
2 & 0 & 5 & 31 & 84 & 131 & 130 & 85 & 36 & 9 & 1 \\ 
3 & 0 & 7 & 64 & 272 & 715 & 1301 & 1730 & 1722 & 1288 & 715 & 286 & \ldots \\% & 78 & 13 & 1 \\ 
4 & 0 & 9 & 109 & 632 & 2338 & 6188 & 12418 & 19496 & 24337 & 24318 & 19449 & \ldots \\% & 12376 & 6188 & 2380 & 680 & 136 & 17 & 1 \\ 
5 & 0 & 11 & 166 & 1220 & 5820 & 20217 & 54264 & 116412 & 203655 & 294040 & 352760 & \ldots \\% & 352726 & 293931 & 203490 & 116280 & 54264 & 20349 & 5985 & 1330 & 210 & 21 & 1 \\ 
\end{array}$ 
\caption{$ \cof{x^k} \left((x+1)^{4n+1}+\left(x^2-1\right)(x+1)^{2n}\right) $ (cf.\ \textbf{R}\ref{tangle:A26}).}
\label{Tab:T64}		
\end{table}

\begin{table}[H] 
\centering 
$\begin{array}{@{}c|rrrrrrrrrrrr@{}}%
n \backslash k &0 & 1 & 2 & 3 & 4 & 5 & 6 & 7 & 8 & 9 & 10 & 11 \\ 
\midrule 
0 & 0 & 0 & 1 \\ 
1 & 0 & 1 & 4 & 6 & 4 & 1 \\ 
2 & 0 & 4 & 23 & 56 & 75 & 60 & 29 & 8 & 1 \\ 
3 & 0 & 12 & 98 & 355 & 750 & 1022 & 938 & 588 & 250 & 70 & 12 & \ldots \\% & 1 \\ 
4 & 0 & 32 & 344 & 1688 & 4999 & 9952 & 14037 & 14400 & 10854 & 6000 & 2402 & \ldots \\% & 680 & 131 & 16 & 1 \\ 
5 & 0 & 80 & 1072 & 6680 & 25670 & 68011 & 131550 & 191840 & 214720 & 185955 & 124652 & \ldots \\% & 64240 & 25094 & 7265 & 1510 & 216 & 20 & 1 \\ 
\end{array}$
\caption{$ \cof{x^k} \left(\left(2x^3+6x^2+6x+2\right)^n+\left(x^2-1\right)\left(x^3+4x^2+5x+2\right)^n\right) $ (cf.\ \textbf{N}\ref{tangle:A27}).}
\label{Tab:T65}		 
\end{table}

\begin{table}[H] 
\centering 
$\begin{array}{@{}c|rrrrrrrrrrr@{}}%
n \backslash k &0 & 1 & 2 & 3 & 4 & 5 & 6 & 7 & 8 & 9 & 10 \\ 
\midrule 
0 & 0 & 1 & 1 \\ 
1 & 0 & 3 & 10 & 12 & 6 & 1 \\ 
2 & 0 & 8 & 47 & 116 & 155 & 120 & 53 & 12 & 1 \\ 
3 & 0 & 20 & 170 & 643 & 1422 & 2030 & 1946 & 1260 & 538 & 142 & \ldots \\% & 20 & 1 \\ 
4 & 0 & 48 & 536 & 2744 & 8519 & 17872 & 26709 & 29184 & 23526 & 13920 & \ldots \\% & 5922 & 1736 & 323 & 32 & 1 \\ 
5 & 0 & 112 & 1552 & 10040 & 40230 & 111691 & 227646 & 352000 & 420640 & 391875 & \ldots \\% & 284812 & 160336 & 68774 & 21825 & 4870 & 696 & 52 & 1 \\ 
\end{array}$ 
\caption{$ \cof{x^k} \left((x+1)\left(2x^3+6x^2+6x+2\right)^n+\left(x^2-1\right)\left(x^3+4x^2+5x+2\right)^n\right) $ (cf.\ \textbf{R}\ref{tangle:A27}).}
\label{Tab:T66}	
\end{table}

\begin{table}[H]
\centering
$\begin{array}{@{}c|rrrrrrrrrrr@{}}%
n \backslash k & 0 & 1 & 2 & 3 & 4 & 5 & 6 & 7 & 8 & 9 & 10\\
\midrule
0 & 0 & 0 & 1\\
1 & 0 & 2 & 4 & 2\\
2 & 0 & 8 & 24 & 24 & 8\\
3 & 0 & 24 & 104 & 176 & 144 & 56 & 8\\
4 & 0 & 64 & 368 & 896 & 1200 & 960 & 464 & 128 & 16\\
5 & 0 & 160 & 1152 & 3680 & 6880 & 8352 & 6880 & 3872 & 1440 & 320 & 32\\
\end{array}$
\caption{$ \cof{x^k} \left(\left(2x^2+4x+2\right)^n+\left(x^2-1\right)(2x+2)^n\right) $ (cf.\ \textbf{N}\ref{tangle:A28}).}
\label{Tab:T67}	
\end{table}

\begin{table}[H] 
\centering 
$\begin{array}{@{}c|rrrrrrrrrrrr@{}}%
n \backslash k &0 & 1 & 2 & 3 & 4 & 5 & 6 & 7 & 8 & 9 & 10 & 11 \\ 
\midrule 
0 & 0 & 1 & 1 \\ 
1 & 0 & 4 & 8 & 4 \\ 
2 & 0 & 12 & 40 & 48 & 24 & 4 \\ 
3 & 0 & 32 & 152 & 296 & 304 & 176 & 56 & 8 \\ 
4 & 0 & 80 & 496 & 1344 & 2096 & 2080 & 1360 & 576 & 144 & 16 \\ 
5 & 0 & 192 & 1472 & 5120 & 10720 & 15072 & 14944 & 10592 & 5280 & 1760 & 352 & 32 \\ 
\end{array}$ 
\caption{$ \cof{x^k} \left((x+1)\left(2x^2+4x+2\right)^n+\left(x^2-1\right)(2x+2)^n\right) $ (cf.\ \textbf{R}\ref{tangle:A28}).}
\label{Tab:T68}	
\end{table}

\begin{table}[H] 
\centering 
$\begin{array}{@{}c|rrrrrrrrrrr@{}}%
n \backslash k &0 & 1 & 2 & 3 & 4 & 5 & 6 & 7 & 8 & 9 & 10  \\ 
\midrule 
0 & 0 & 0 & 1 \\ 
1 & 0 & 4 & 8 & 4 \\ 
2 & 0 & 16 & 60 & 88 & 64 & 24 & 4 \\ 
3 & 0 & 48 & 272 & 688 & 1032 & 1016 & 672 & 288 & 72 & 8 \\ 
4 & 0 & 128 & 976 & 3520 & 8000 & 12736 & 14800 & 12672 & 7920 & 3520 & \ldots \\% & 1056 & 192 & 16 \\ 
5 & 0 & 320 & 3072 & 14400 & 43840 & 96384 & 160320 & 205952 & 205920 & 160160 & \ldots \\% & 96096 & 43680 & 14560 & 3360 & 480 & 32 \\ 
\end{array}$ 
\caption{$ \cof{x^k} \left(\left(2x^3+6x^2+6x+2\right)^n+\left(x^2-1\right)(2x+2)^n\right) $ (cf.\ \textbf{N}\ref{tangle:A29}).}
\label{Tab:T69}		
\end{table}

\begin{table}[H] 
\centering 
$\begin{array}{@{}c|rrrrrrrrrrr@{}}%
n \backslash k &0 & 1 & 2 & 3 & 4 & 5 & 6 & 7 & 8 & 9 & 10\\ 
\midrule 
0 & 0 & 1 & 1 \\ 
1 & 0 & 6 & 14 & 10 & 2 \\ 
2 & 0 & 20 & 84 & 148 & 144 & 84 & 28 & 4 \\ 
3 & 0 & 56 & 344 & 976 & 1704 & 2024 & 1680 & 960 & 360 & 80 & \ldots \\% & 8 \\ 
4 & 0 & 144 & 1168 & 4576 & 11520 & 20656 & 27472 & 27456 & 20592 & 11440 & \ldots \\% & 4576 & 1248 & 208 & 16 \\ 
5 & 0 & 352 & 3552 & 17760 & 58400 & 140064 & 256416 & 366112 & 411840 & 366080 & \ldots \\% & 256256 & 139776 & 58240 & 17920 & 3840 & 512 & 32 \\ 
\end{array}$
\caption{$ \cof{x^k} \left((x+1)\left(2x^3+6x^2+6x+2\right)^n+\left(x^2-1\right)(2x+2)^n\right) $  (cf.\ \textbf{R}\ref{tangle:A29}).}
\label{Tab:T70}		 
\end{table}

\begin{table}[H]
\centering
$\begin{array}{@{}c|rrrrrrrrrrrr@{}}%
n \backslash k & 0 & 1 & 2 & 3 & 4 & 5 & 6 & 7 & 8 & 9 & 10 & 11\\
\midrule
0 &  0 &  0 &  1\\
1 &  0 &  2 &  6 &  6 &  2\\
2 &  0 &  16 &  60 &  88 &  64 &  24 &  4\\
3 &  0 &  96 &  496 &  1064 &  1224 &  816 &  320 &  72 &  8\\
4 &  0 &  512 &  3456 &  10112 &  16752 &  17216 &  11312 &  4736 &  1232 &  192 &  16\\
5 &  0 &  2560 &  21504 &  80640 &  177600 &  254112 &  246560 &  164256 &  74400 &  22240 &  4192 & \ldots \\% &  480 &  32\\ 
\end{array}$
\caption{$ \cof{x^k} \left((2x+2)^{2n}+\left(x^2-1\right)\left(2x^2+6x+4\right)^n\right) $ (cf.\ \textbf{N}\ref{tangle:A30}).}
\label{Tab:T71}		
\end{table}	

\begin{table}[H] 
\centering 
$\begin{array}{@{}c|rrrrrrrrrrr@{}}%
n \backslash k &0 & 1 & 2 & 3 & 4 & 5 & 6 & 7 & 8 & 9 & 10\\ 
\midrule 
0 & 0 & 1 & 1 \\ 
1 & 0 & 6 & 14 & 10 & 2 \\ 
2 & 0 & 32 & 124 & 184 & 128 & 40 & 4 \\ 
3 & 0 & 160 & 880 & 2024 & 2504 & 1776 & 704 & 136 & 8 \\ 
4 & 0 & 768 & 5504 & 17280 & 31088 & 35136 & 25648 & 11904 & 3280 & 448 & \ldots \\% & 16 \\ 
5 & 0 & 3584 & 31744 & 126720 & 300480 & 469152 & 504608 & 379296 & 197280 & 68320 & \ldots \\% & 14432 & 1504 & 32 \\ 
\end{array}$ 
\caption{$ \cof{x^k} \left((x+1)(2x+2)^{2n}+\left(x^2-1\right)\left(2x^2+6x+4\right)^n\right) $ (cf.\ \textbf{R}\ref{tangle:A30}).}
\label{Tab:T72}		
\end{table}

\begin{table}[H] 
\centering 
$\begin{array}{@{}c|rrrrrrrrrrrr@{}}%
n \backslash k &0 & 1 & 2 & 3 & 4 & 5 & 6 & 7 & 8 & 9 & 10 & 11\\ 
\midrule 
0 & 0 & 0 & 1 \\ 
1 & 0 & 1 & 4 & 6 & 4 & 1 \\ 
2 & 0 & 2 & 14 & 42 & 70 & 70 & 42 & 14 & 2 \\ 
3 & 0 & 3 & 31 & 145 & 405 & 750 & 966 & 882 & 570 & 255 & 75 & \ldots \\% & 13 & 1 \\ 
4 & 0 & 4 & 55 & 352 & 1391 & 3796 & 7579 & 11440 & 13299 & 12012 & 8437 & \ldots \\% & 4576 & 1885 & 572 & 121 & 16 & 1 \\ 
5 & 0 & 5 & 86 & 700 & 3585 & 12956 & 35120 & 74088 & 124540 & 169390 & 188188 & \ldots \\% & 171600 & 128518 & 78780 & 39200 & 15608 & 4860 & 1141 & 190 & 20 & 1 \\ 
\end{array}$ 
\caption{$ \cof{x^k} \left((x+1)^{4n}+\left(x^2-1\right)(x+1)^{3n}\right) $ (cf.\ \textbf{N}\ref{tangle:A31}).}
\label{Tab:T73}		
\end{table}

\begin{table}[H] 
\centering 
$\begin{array}{@{}c|rrrrrrrrrrrr@{}}%
n \backslash k &0 & 1 & 2 & 3 & 4 & 5 & 6 & 7 & 8 & 9 & 10 & 11 \\ 
\midrule 
0 & 0 & 1 & 1 \\ 
1 & 0 & 2 & 8 & 12 & 8 & 2 \\ 
2 & 0 & 3 & 22 & 70 & 126 & 140 & 98 & 42 & 10 & 1 \\ 
3 & 0 & 4 & 43 & 211 & 625 & 1245 & 1758 & 1806 & 1362 & 750 & 295 & \ldots \\% & 79 & 13 & 1 \\ 
4 & 0 & 5 & 71 & 472 & 1951 & 5616 & 11947 & 19448 & 24739 & 24882 & 19877 & \ldots \\% & 12584 & 6253 & 2392 & 681 & 136 & 17 & 1 \\ 
5 & 0 & 6 & 106 & 890 & 4725 & 17801 & 50624 & 112848 & 202060 & 295360 & 356148 & \ldots \\% & 356356 & 296478 & 204750 & 116720 & 54368 & 20364 & 5986 & 1330 & 210 & 21 & 1 \\ 
\end{array}$ 
\caption{$ \cof{x^k} \left((x+1)^{4n+1}+\left(x^2-1\right)(x+1)^{3n}\right) $ (cf.\ \textbf{R}\ref{tangle:A31}).}
\label{Tab:T74}	
\end{table}

\begin{table}[H] 
\centering 
$\begin{array}{@{}c|rrrrrrrrrrr@{}}%
n \backslash k &0 & 1 & 2 & 3 & 4 & 5 & 6 & 7 & 8 & 9 & 10 \\ 
\midrule 
0 & 0 & 0 & 1 \\ 
1 & 0 & 3 & 7 & 5 & 1 \\ 
2 & 0 & 18 & 66 & 92 & 60 & 18 & 2 \\ 
3 & 0 & 81 & 432 & 972 & 1200 & 886 & 400 & 108 & 16 & 1 \\ 
4 & 0 & 324 & 2295 & 7236 & 13413 & 16264 & 13574 & 7976 & 3306 & 948 & \ldots \\% & 179 & 20 & 1 \\ 
5 & 0 & 1215 & 10773 & 43875 & 108990 & 184863 & 226895 & 208059 & 144820 & 76805 & \ldots \\% & 30819 & 9185 & 1966 & 285 & 25 & 1 \\ 
\end{array}$ 
\caption{$ \cof{x^k} \left(\left(x^3+5x^2+7x+3\right)^n+\left(x^2-1\right)\left(x^2+4x+3\right)^n\right) $ (cf.\ \textbf{N}\ref{tangle:A32}).}
\label{Tab:T75}		
\end{table}

\begin{table}[H] 
\centering 
$\begin{array}{@{}c|rrrrrrrrrrr@{}}%
n \backslash k &0 & 1 & 2 & 3 & 4 & 5 & 6 & 7 & 8 & 9 & 10\\ 
\midrule 
0 & 0 & 1 & 1 \\ 
1 & 0 & 6 & 14 & 10 & 2 \\ 
2 & 0 & 27 & 108 & 171 & 136 & 57 & 12 & 1 \\ 
3 & 0 & 108 & 621 & 1548 & 2200 & 1972 & 1162 & 452 & 112 & 16 & \ldots \\% & 1 \\ 
4 & 0 & 405 & 3051 & 10422 & 21417 & 29591 & 29038 & 20772 & 10898 & 4147 & \ldots \\% & 1111 & 198 & 21 & 1 \\ 
5 & 0 & 1458 & 13608 & 59130 & 159165 & 297558 & 409782 & 429334 & 347815 & 218990 & \ldots \\% & 106764 & 39830 & 11131 & 2250 & 310 & 26 & 1 
\end{array}$ 
\caption{$ \cof{x^k} \left((x+1)\left(x^3+5x^2+7x+3\right)^n+\left(x^2-1\right)(x^2+4x+3)^n\right) $ (cf.\ \textbf{R}\ref{tangle:A32}).}
\label{Tab:T76}	
\end{table}

\begin{table}[H]
\centering
$\begin{array}{@{}c|rrrrrrrrrrr@{}}%
n \backslash k & 0 & 1 & 2 & 3 & 4 & 5 & 6 & 7 & 8 & 9 & 10\\
\midrule
0 & 0 & 0 & 1\\
1 & 0 & 2 & 6 & 6 & 2\\
2 & 0 & 8 & 40 & 80 & 80 & 40 & 8\\
3 & 0 & 24 & 176 & 560 & 1008 & 1120 & 784 & 336 & 80 & 8\\
4 & 0 & 64 & 624 & 2752 & 7248 & 12672 & 15456 & 13440 & 8352 & 3648 & \ldots \\% & 1072 & 192 & 16\\
5 & 0 & 160 & 1952 & 11040 & 38400 & 91872 & 160160 & 210144 & 211200 & 163680 & \ldots \\% & 97504 & 44000 & 14592 & 3360 & 480 & 32
\end{array}$
\caption{$ \cof{x^k} \left(\left(2x^3+6x^2+6x+2\right)^{n} +\left(x^2-1\right)\left(2x^2+4x+2\right)^n\right) $ (cf.\ \textbf{N}\ref{tangle:A33}).}
\label{Tab:T77}		
\end{table}	

\begin{table}[H] 
\centering 
$\begin{array}{@{}c|rrrrrrrrrrr@{}}%
n \backslash k &0 & 1 & 2 & 3 & 4 & 5 & 6 & 7 & 8 & 9 & 10 \\ 
\midrule 
0 & 0 & 1 & 1 \\ 
1 & 0 & 4 & 12 & 12 & 4 \\ 
2 & 0 & 12 & 64 & 140 & 160 & 100 & 32 & 4 \\ 
3 & 0 & 32 & 248 & 848 & 1680 & 2128 & 1792 & 1008 & 368 & 80 & \ldots \\% & 8 \\ 
4 & 0 & 80 & 816 & 3808 & 10768 & 20592 & 28128 & 28224 & 21024 & 11568 & \ldots \\% & 4592 & 1248 & 208 & 16 \\ 
5 & 0 & 192 & 2432 & 14400 & 52960 & 135552 & 256256 & 370304 & 417120 & 369600 & \ldots \\% & 257664 & 140096 & 58272 & 17920 & 3840 & 512 & 32 \\ 
\end{array}$ 
\caption{$ \cof{x^k} \left((x+1)\left(2x^3+6x^2+6x+2\right)^{n} +\left(x^2-1\right)\left(2x^2+4x+2\right)^n\right) $ (cf.\ \textbf{R}\ref{tangle:A33}).}
\label{Tab:T78}		
\end{table}

\begin{table}[H]
\centering
$\begin{array}{@{}c|lrrrrrrr@{}}%
n \backslash k & 0 & 1 & 2 & 3 & 4 & 5 & 6\\
\midrule
0 & 0 & 1\\
1 & 0 & 0 & 1\\
2 & 0 & 0 & 0 & 1\\
3 & 0 & 0 & 0 & 0 & 1\\
4 & 0 & 0 & 0 & 0 & 0 & 1\\
5 & 0 & 0 & 0 & 0 & 0 & 0 & 1\\
\end{array}$
\caption{$ \cof{x^k}x^{n+1} $ (cf.\ \textbf{D}\ref{tangle:A35}) \cite[\seqnum{A129185}]{Sloane}.}
\label{Tab:T79}
\end{table}

\begin{table}[H]
\centering
$\begin{array}{@{}c|lrrrrrr@{}}%
n \backslash k & 0 & 1 & 2 & 3 & 4 & 5\\
\midrule
0 & 0 & 0 & 1\\
1 & 0 & 1 \\
2 & 0 & 0 & 1\\
3 & 0 & 0 & 0 & 1\\
4 & 0 & 0 & 0 & 0 & 1\\
5 & 0 & 0 & 0 & 0 & 0 & 1\\
\end{array}$
\caption{$ \cof{x^k}p_n(x) $, where $ p_0(x) = x^2 $ and $ p_n(x)= x^n $ for $ n\geq 1 $ (cf.\ \textbf{N}\ref{tangle:A35}).}
\label{Tab:T80}
\end{table}

\begin{table}[H]
\centering
$\begin{array}{@{}c|lrrrrrrr@{}}%
n \backslash k & 0 & 1 & 2 & 3 & 4 & 5 & 6\\
\midrule
0 & 0 & 1 & 1\\
1 & 0 & 1 & 1\\
2 & 0 & 0 & 1 & 1\\
3 & 0 & 0 & 0 & 1 & 1\\
4 & 0 & 0 & 0 & 0 & 1 & 1\\
5 & 0 & 0 & 0 & 0 & 0 & 1 & 1\\
\end{array}$
\caption{$ \cof{x^k}p_n(x) $, where $ p_0(x) = x^2+x $ and $ p_n(x)= x^{n+1}+x^n $ for $ n\geq 1 $ (cf.\ \textbf{R}\ref{tangle:A35}).}
\label{Tab:T81}
\end{table}

\bigskip
\hrule
\bigskip

\noindent 2010 {\it Mathematics Subject Classification}: 57M25; 05A10.

\bigskip
\hrule
\bigskip

%http://oeis.org/A300453http://oeis.org/A034870http://oeis.org/A038208

\noindent (Concerned with sequences
\seqnum{A007318}, \seqnum{A034870}, \seqnum{A038208}, \seqnum{A129185}, \seqnum{A139548}, \seqnum{A299989},\\ \seqnum{A300184}, \seqnum{A300192} and \seqnum{A300453}.)
\bigskip
\hrule
\bigskip
\end{document}